\documentclass[12pt,thmsa, reqno]{amsart}
\usepackage{amsmath,amssymb,latexsym,graphicx}
\usepackage{lineno}
\usepackage{enumerate}

\input{amssym.def}
\input{amssym.tex}
\usepackage{amsmath}

\def\Cal{\mathcal}

\def\P{{\Cal P}}

\def\bbr{{\Bbb R}}

\def\bbc{{\Bbb C}}

\def\supp{{\hbox{\rm supp}}}

\def\const{{\hbox{\rm const}}}

\def\Aut{{\hbox{\rm Aut}}}

\def\det{{\hbox{\rm det}}}

\def\part{\partial}

\def\vp{\varphi}

\def\z{\zeta}

\newtheorem{theorem}{Theorem}[section]
\newtheorem{lemma}[theorem]{Lemma}

\theoremstyle{definition}
\newtheorem{definition}[theorem]{Definition}
\newtheorem{example}[theorem]{Example}

\theoremstyle{remark}
\newtheorem{remark}[theorem]{Remark}

\theoremstyle{corollary}
\newtheorem{corollary}[theorem]{Corollary}
\newtheorem{proposition}[theorem]{Proposition}

\numberwithin{equation}{section}

\newcommand{\be}{\begin{equation}}
\newcommand{\ee}{\end{equation}}

\newcommand{\bea}{\begin{eqnarray}}
\newcommand{\eea}{\end{eqnarray}}
\newcommand{\Bea}{\begin{eqnarray*}}
\newcommand{\Eea}{\end{eqnarray*}}

\def\sideremark#1{\ifvmode\leavevmode\fi\vadjust{\vbox to0pt{\vss
 \hbox to 0pt{\hskip\hsize\hskip1em
\vbox{\hsize2cm\tiny\raggedright\pretolerance10000
 \noindent #1\hfill}\hss}\vbox to8pt{\vfil}\vss}}}%

\begin{document}
\title[]
{Injectivity of pairs of non-central Funk transforms }

\author{Mark \  Agranovsky}

\address{Bar-Ilan University,Department of Mathematics, Ramat-Gan, 5290002; \ Holon Institute of Technology, Holon, 5810201}
\email{agranovs@math.biu.ac.il }





\begin{abstract}  We study  Funk-type transforms  on the unit sphere  in $\bbr^{n}$ associated with cross-sections of the sphere by lower-dimensional planes  passing through an arbitrary  fixed  point inside the sphere or outside. Our main concern is injectivity of the corresponding paired transforms generated by two families of planes centered at distinct points. Necessary and sufficient conditions for the paired transforms to be injective are obtained,  depending on geometrical configuration of the centers. Our method relies on the action of the automorphism group of the unit ball and the relevant billiard-like dynamics on the sphere.
 \end{abstract}

\maketitle

\section{Introduction}

The classical Funk transform  and its higher dimensional generalizations integrate functions on the unit sphere  $S^{n-1}$ in $\mathbb R^n$ over the great subspheres, obtained by intersection of $S^{n-1}$ with planes of fixed dimension passing through the origin    \cite{Fu13}, \cite {H11}, \cite{Gi}, \cite{Ru15}. These transforms have applications in geometric tomography \cite{Ga}, medical imaging \cite{Tuch}.
The kernel of Funk transforms consists  of odd functions and  inversion formulas,  recovering the even part of functions, are known.

Recently, a shifted, non-central,  Funk transform, where the center (i.e. the common point of intersecting planes)   differs from the  origin, has appeared in the focus of researchers  \cite{Sa16}, \cite{Sa17}, \cite{Q}, \cite{Q1}, \cite{Ru18}, \cite{Pala1}, \cite{AgR}. Main results there address the description of the kernel and inversion formulas in the case  when the center lies strictly inside  the sphere. Similar questions for  exterior center  are studied in  \cite{AgR1}.

While complete recovery of functions from a single shifted Funk transform is impossible due to the nontrivial kernel, it was proved in  \cite{AgR}, that
the data provided by Funk transforms with two distinct centers inside $S^{n-1}$ are sufficient for the unique recovery. We call the transform defined by
a pair of shifted Funk transforms {\it paired shifted Funk transform}. This definition is applicable  to an arbitrary pair of distinct centers in $\mathbb R^n.$

In the present article  we generalize the results from \cite{AgR} and extend the single-center results from \cite{AgR1} to the paired Funk transforms with arbitrary centers, each of which can be either inside or outside  the unit sphere. Here  the Funk transforms centered on the surface of the sphere are excluded because such transforms are injective (see \cite{AD} \cite [p. 145] {H11}, \cite[Section 7.2]{Ru15}) and the additional center is not needed. It turns out that the injectivity of the paired shifted Funk transform essentially  depends on the mutual location of the centers. We obtain necessary and sufficient geometric conditions, under which the  location of the centers  provides injectivity of the relevant paired transform.

The approach relies on the action of the group $\Aut(B^n)$  of automorphisms of the unit ball and exploits group-invariance arguments.
On one hand, it yields new, performed in an invariant form and not demanding cumbersome coordinate computations, proofs of main formulas for the single-centered  Funk transforms obtained in \cite{AgR, AgR1}. On the other hand, groups of M\"obius transformations, intimately related to $\Aut(B^n),$ naturally appear in the description of the kernel of the paired Funk transform. We think that the developed group-theoretical approach might be useful in other similar problems.

\section{Setting of the problem, main results and outline of the approach}

\subsection{Basic notation}
We will be dealing with the real Euclidean space $\mathbb R^{n}$ equipped with the standard norm $|x|$ and  the inner product $\langle x, y \rangle.$ The open unit ball in $\mathbb R^n$ will be denoted $B^n$ and  its boundary by  $S^{n-1}.$ Throughout the article,  we fix a natural number $1 \leq k < n.$
We denote $Gr(n,k)$ be the Grassmann manifold of all  $k$-dimensional affine planes in $\mathbb R^n.$
Given a point $a \in \mathbb R^n,$ the notation $Gr_a(n,k)$ stands for the submanifold of all affine $k$-planes containing $a.$  In particular,
$Gr_0(n, k)$ denotes the manifold of all $k$-dimensional linear subspaces of $\mathbb R^n.$  Unit linear operators in corresponding spaces is denoted  $I,$ while $id$ stands for identical mappings.  Given a  mapping $T : S^{n-1} \to S^{n-1}$ we denote the $q$-th iteration $T^{\circ q}= T \circ ... \circ T.$

\subsection{Setting of the problem}

For  $ f\in C(S^{n-1}) $ and $ E \in Gr_a(n,k)$, we  define
\begin{equation}\label {E:Fa-def} (F_a f) (E)=\int\limits_{ S^{n-1} \cap E} f(x) dA_{E},
\end{equation}
where $dA_{E}$ is the surface area measure on the $(k-1)$-dimensional sphere $S^n \cap E.$
 The operator $F_a$ takes functions on $S^{n-1}$ to functions on $Gr_a(n,k)$. We call it {\it the (shifted) Funk
transform} with  center $a.$ The case $|a|=1$ is well studied (see Introduction) and will be excluded from our consideration.

Every operator $F_a$ with $|a| \neq 1$ has a nontrivial kernel, so that a  function $f \in C(S^n)$ cannot be recovered from the single Funk data $g_a=F_af.$
 It is natural to ask whether we can
 recover  $f$  from the pair
of two equations  $g_a=F_af$ and  $g_b=F_bf$, if $a$ and $b$ are distinct centers not belonging to $S^{n-1}$? More precisely, we have the following

{\bf Question.}
{\it For what pairs $(a,b)$ with $a,b \notin S^{n-1}$ is the  paired Funk transform
\begin{equation}\label{E:paired}
(F_a, F_b): f \to (F_a f, F_b f), \qquad  f \in C(S^{n-1}),
\end{equation}
injective,
i.e.,   $\ker (F_a, F_b)=\ker F_a \cap \ker F_b =\{0\} ? $}

\subsection{Main results}

The answer to this question is the main result of the paper. We will present three equivalent formulations.

To formulate the first version we need to define a self-mapping of $S^{n-1}$ associated with a pair points  $a, b.$
\begin{definition}\label{D:T-def}
Let  $a, b \in \mathbb R^n.$ Define the "$V$-like" mapping $T: S^{n-1} \to S^{n-1}$ as follows. Given $x \in S^{n-1},$
let $[x, \ x^{\prime}, \ x^{\prime\prime} ]= [x, \ x^{\prime}] \cup [x^{\prime}, \ x ^{\prime\prime}]$ be the $V$-like broken line such that
\begin{enumerate}[(\rm i)]
\item the vertices $x, \ x^{\prime}, \ x ^{\prime\prime} \in S^{n-1},$
\item  $x^{\prime}$ belongs to the straight line through $x$ and $a,$
\item $x^{\prime\prime}$ belongs to the straight line through $x^{\prime}$ and $b.$
\end{enumerate}
Then set $Tx:= x^{\prime\prime}.$
\end{definition}

\begin{figure}[h]
\centering
\includegraphics[width=9cm]{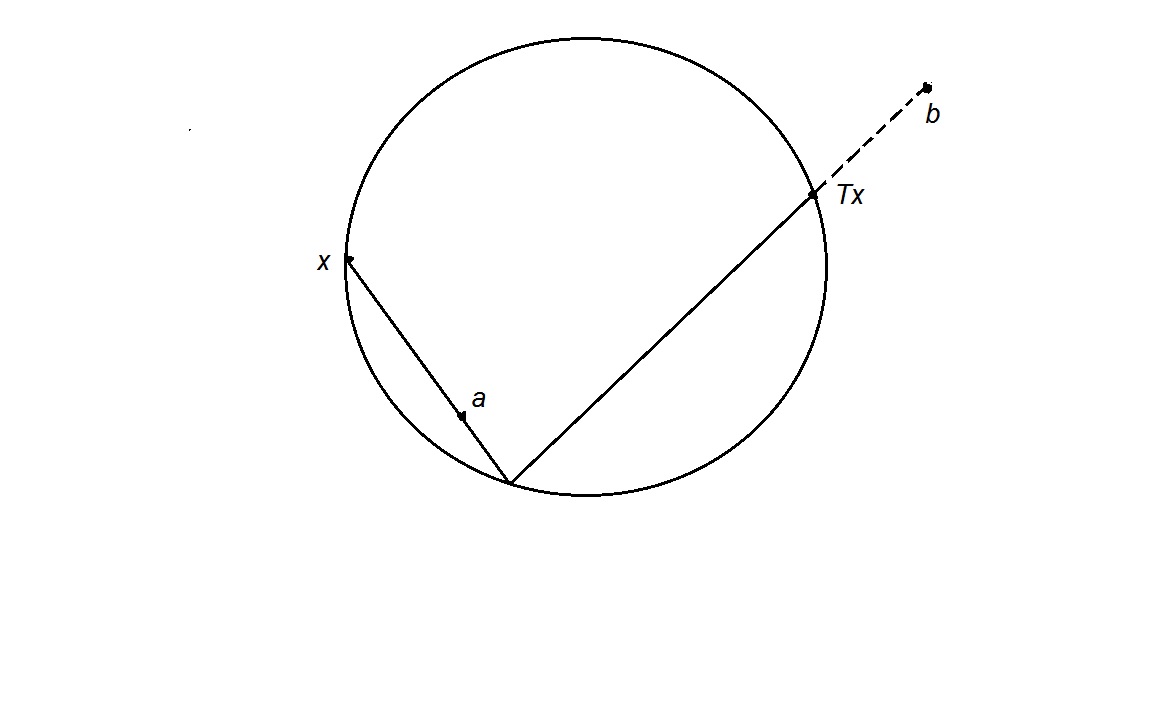}
\caption{The $V$-mapping $T$}
\end{figure}

\begin{theorem}\label{T:main0}
The paired  transform (\ref{E:paired}) is injective, i.e.,
$\ker (F_a, F_b)=\ker F_a \cap \ker F_b = \{0\}$ if and only  if the $V$-like mapping $T : S^{n-1} \to S^{n-1}$ is non-periodic, i.e.,
$T^{\circ q} \neq id$ for any $ q \in \mathbb N.$
\end{theorem}

Theorem \ref{T:main0} possesses the following equivalent reformulation. Given a point $a \in \mathbb R^n,$ denote $\tau_a x \in S^{n-1}$ such point
that the segment $[x,\tau_a x]$ belongs to the line through $x$ and $a.$  It is clear that the mapping $T$ decomposes as $Tx=\tau_b (\tau_a x).$
\begin{corollary}\label{C:finite}
The paired transform (\ref{E:paired}) fails to be injective, i.e. $\ker (F_a, F_b)=\ker F_a \cap \ker F_b \neq \{0\}$ if and only if
the group $G$ generated by the two mappings $\tau_a, \ \tau_b: S^{n-1} \to S^{n-1}$ is finite.
\end{corollary}

To present an analytic form of Theorem \ref{T:main0},
 we set
\be\label {theta}
\Theta(a,b)= \frac{\langle  a , b \rangle -1 }{ \sqrt {(1-|a|^2)(1-|b|^2)} }, \ee
with the principal branch of the square root. The number $\Theta(a,b)$ can be either real or pure imaginary. The latter holds if and only if
\be\label {theta1} \langle a, b \rangle \neq 1\quad \text {\rm and}\quad  (1-|a|^2)(1-|b|^2) <0.\ee
The second inequality means  that $a$ and $ b$ are separated by the unit sphere $S^{n-1}.$
If $\Theta(a,b)$ is real-valued and belongs to $[-1,1],$  the angle is well defined:
$$\theta_{a,b}=\arccos \Theta(a,b).$$
The  ratio
\begin{equation} \label{E:kappa}
 \kappa(a,b)=\frac{\theta_{a,b}}{\pi}
 \end{equation}
is called the {\it rotation number}. For large $|a|$ and $|b|$, the number $\kappa(a,b)$ is close to the angle between the vectors $a$ and $b$, divided by $\pi.$

\begin{theorem} \label{T:main}
The paired  transform (\ref{E:paired}) is non-injective, i.e., \\
$\ker (F_a, F_b)=\ker F_a \cap \ker F_b  \neq \{0\}$ if and only  if
\begin{enumerate}[(\rm i)]
\item
$\Theta(a,b) \in [-1,1]$ and
\item the rotation number $\kappa(a,b)$ is rational.
\end{enumerate}
\end{theorem}
Finally, a geometric version of the main result  reads as follows.
\begin{theorem}\label{T:main-geom} The paired  transform (\ref{E:paired}) is injective if and only if
\begin{enumerate}[(\rm i)]
\item $\langle a , b \rangle \neq 1$ and
\item either the straight line $L_{a,b}$ through $a$ and $b$ meets $S^{n-1}$ or, otherwise, $\Theta(a,b) \in [-1,1]$ and the
rotation number$\kappa(a,b)$ is irrational.
\end{enumerate}
\end{theorem}

\begin{example} The paired  transform $(F_a, F_{a^*})$ with $a^*=a/|a|^2$, $a\neq 0$, is non-injective, because  $\langle a, a^* \rangle =1.$

\end{example}

\begin{example} The paired  transform  $(F_a, F_b)$ with at least one interior center, $a$ or $b,$ is injective because in this case  $L_{a,b} \cap S^{n-1} \neq \emptyset.$ For two interior centers $a, b \ in B^n,$ the injectivity was proved earlier in \cite{AgR}.

\end{example}

\subsection {Plan of the paper and outline of the approach}
Section 3 contains preliminaries.
In Sections 4-6 we establish a link between Funk transforms with different centers, by means of the action on $S^{n-1}$ of the group $\Aut(B^n)$ of fractional-linear automorphisms of the unit ball $B^n.$ This group is associated with the hyperbolic space structure on $B^n$ and is intimately related to the group of M\"obius transformations of $B^n.$.
The above strategy is in line with the concept of factorable mappings in (\cite{Pala}, Chapter 3). Using group invariance argument, we construct intertwining operators between the shifted transforms and standard ones. This leads to characterization the kernels of each single  transform $F_a$ and $F_b,$  in terms of certain symmetries with respect to the centers $a$ and $b.$ Such characterizations were earlier established in \cite{Q}, \cite{AgR},\cite{AgR1} by different methods.
In Sections 7 we study the dynamical system generated by the above symmetries of the unit sphere. More precisely, the composition of the aforementioned symmetries produces the billiard-like self-mapping $T$ of $S^{n-1},$ defined in Definition \ref{D:T-def}, such that the space of $T$-automorphic functions contains the common kernel $ker F_a \cap ker F_b.$ The mapping $T$ generates complex M\"obius transformations of 2-dimensional cross-sections of $S^{n-1}.$ Then the description of the common kernel (Theorem \ref{T:main}) follows from classification of the types of orbits of the dynamical system generated by iterations of those M\"obius transformations. In particular, the non-trivial kernel corresponds to periodic mappings $T$. The proofs of main results are given in Section 8. Section 9 is devoted to some generalizations, in particular, to discussion of the case of arbitrary finite sets of Funk transforms.
Some open questions are formulated. Section 10 contains concluding remarks.

\section{Preliminaries}

\subsection{The group $\Aut(B^n)$ } \label{S:Aut}

We start with the description of the group which is behind all our main constructions. The group  $\Aut(B^n)$ can be defined in many equivalent ways.
In the context of this article, it will be convenient to define this group
as a real version of the  group $\Aut(B^n_{\mathbb C})$  of  biholomorphic automorphisms of the open unit ball  $B^n_{\mathbb C}$  in  $\mathbb C^n=\mathbb R^n+ i \mathbb R^n$ ( see, e.g., \cite[Chapter 2]{Rudin}), if we consider $B^n$ as the real part of $B^n_{\mathbb C}$.

\begin{definition} \label{D:Aut} {\it Define $\Aut(B^n)$ as  the restriction onto  $B^n = B^n_{\mathbb C} \cap \mathbb R^n$  of the subgroup
of all complex automorphisms of $B^n_{\mathbb C}$ preserving $B^n$.}
\end{definition}

By this definition, every element of $\Aut(B^n)$ extends as a holomorphic self-mapping of $B^n_{\mathbb C}$ and, by the uniqueness theorem for holomorphic functions, this extension is unique.  Further, each  automorphism from $\Aut(B_{\mathbb C}^n)$ extends continuously onto the closed ball $\bar B^n_{\mathbb C}$. Hence  all automorphisms in $\Aut(B^n)$ extend continuously onto  $\bar B^n$.

An important representative of the group
 $\Aut (B^n)$ is an involutive automorphism
\begin{equation}\label{E:phi-def}
\varphi_a (x)=\frac{a - P_a x - \sqrt{1-|a|^2}\, Q_a x}{ 1- \langle x, a \rangle }, \qquad a\in B^n.
\end{equation}
Here  $P_a x = \langle a, x \rangle a/|a|^2$   if $ a \neq 0$, and   $Q_a x=x - P_a x$. If $a=0$ we simply set $\vp_0 (x)=-x$. The operator $P_a$ is an orthogonal projection onto the space $[a]$  spanned by the vector $a$, $Q_a$ is an orthogonal projection onto the complementary subspace $a^\perp$.

Complex analogs $\varphi_a(z)$, with $a, \ z\in B^n_{\mathbb C} $, of  (\ref{E:phi-def}) are investigated in \cite[Chapter 2]{Rudin}. Main properties of $\varphi_a(x)$ are inherited from $\varphi_a(z)$ or can be checked by straightforward computation.

\begin{lemma}\label{L:phi-prop} {\rm ( \cite[Theorem 2.2.2]{Rudin})}
\begin{enumerate}[(\rm i)]
\item $\varphi_a\big( \varphi_a(x) \big) =x$ for all  $x\in \bar B^n$.
\item $\varphi_a(0)=a$, $\ \varphi_a(a)=0.$
\item For all  $x, y \in \bar B^n$ satisfying $\langle x, a \rangle \neq 1$, $\langle y, a \rangle \neq 1$,
\be\label {three} 1- \langle \varphi_a (x), \varphi_a (y) \rangle= \frac {(1-|a|^2)( 1-\langle x, y \rangle ) }{(1- \langle x, a \rangle)(1- \langle y, a \rangle) }.\ee
In particular, if $\langle x, a \rangle \neq 1$, then
\be\label {three1} 1- |\varphi_a(x)|^2=\frac{(1-|a|^2)(1-|x|^2)}{(1- \langle  x, a \rangle )^2}.\ee
\item $\varphi_a(B^n)=B^n$, $\ \varphi_a (S^{n-1})=S^{n-1}.$
\item $\varphi_a$ maps affine subsets of $\bar B^n$ (intersections of affine subspaces in $\mathbb R^n$ with the unit ball) onto affine subsets.
\end{enumerate}
\end{lemma}

The following lemma is a real version of the polar decomposition $g=\varphi_a U, \ a \in B^n_{\mathbb C}, \ U \in U(n)$ of complex automorphisms
$g \in \Aut(B^n_{\mathbb C}) $ \cite[Theorem 2.2.5]{Rudin}. Since $g$ preserves the real ball $B^n$ if and only if $a$ and $U$ are real, we have
\begin{lemma} \label{L:phiU}    Every automorphism $ g \in \Aut(B^n)$ can be uniquely represented as
\[g=U\varphi_a =\varphi_b  V, \qquad   a=g^{-1}(0), \qquad b=g (0), \]
for some $U$ and $V$ in the orthogonal group $O(n)$.
\end{lemma}
The second representation, $g=\varphi_b V$, follows from the first one, applied to $g^{-1}.$

\begin{corollary}
\begin{enumerate}[(i)]
\item The group $\Aut(B^n)$ acts on $B^n$ transitively, because $\varphi_b (\varphi_a (a))=b.$  Moreover, $\Aut(B^n)$ acts transitively on the Grassmanian $Gr_{B^n}(n, k)$ of all $k$-dimensional affine subsets  of $B^n$ (intersections of $B^n$ with affine $k$-planes).
\item The group $O(n)$ is the isotropy group of the origin $0$,
so that the unit ball can be viewed as the homogeneous space
$$B^n = \Aut(B^n) / O(n).$$
\end{enumerate}
\end{corollary}

\subsection{The groups $\Aut(B^n)$ and $M(B^n)$}
The group $\Aut(B^n)$ is closely related with the group $M(B^n)$ of M\"obius transformations which is behind many constructions throughout the article (see Sections 7,8). Recall, that the group $M(B^n)$ consists of finite compositions of inversions (or reflections) of $\mathbb R^n$ in hyperplanes and spheres, orthogonal to $S^{n-1}$
(cf. \cite{Stoll}, \cite {Beardon}). This group maps the ball $B^n$ onto itself and preserves the boundary sphere $S^{n-1}.$

\begin{lemma} The actions of the groups $\Aut(B^n)$ and  $M(B^n)$ on the unit sphere coincide,  $\Aut(B^n)\vert_{S^{n-1}}=M(B^n)\vert_{S^{n-1}}.$
\end{lemma}
\begin{proof}  For $|a|<1$, we denote
$$\psi_a(x)=\frac{a |x-a|^2+ (1-|a|^2)(a-x)}{ |a|^2 |a^*-x|^2},\qquad a^*=\frac{a}{|a|^2}; $$
cf.  \cite[formula (2.1.6)]{Stoll}, where this function is denoted by $\varphi_a (x)$ (in our text,   $\varphi_a (x)$ has  different meaning).
Every element $g \in M(B^n)$ has the form  $g=U \psi_a$ for some $U \in O(n)$ and  $a\in B^n$; see  \cite[Theorem 2.1.2 (b)]{Stoll}.
 The straightforward computation gives
$$\psi_a(x)=\varphi_{\gamma( a) }(x),\qquad \gamma(a)= \frac{2a}{1+|a|^2}, \qquad x\in S^{n-1},$$
where $\varphi_{\gamma( a) }$ is an involution from $\Aut(B^n),$ defined by  (\ref{E:phi-def}) with $a$ replaced by $\gamma(a)$.
 It follows that $\psi_a \vert _{S^{n-1}}  \in \Aut(B^n) \vert_{S^{n-1}}$ and hence, by Lemma \ref{L:phiU}, each element $g=U \psi_a \in M(B^n)$ coincides on $S^{n-1}$ with an element of $\Aut(B^n).$

Conversely, if $ g \in \Aut(B^n)$, then, by Lemma \ref{L:phiU}, $g=U \varphi_{b}$  for some $U \in O(n)$ and  $b\in B^n$.  Setting
\[
b=\gamma(a), \qquad a=\frac{b}{1+\sqrt{1-|b|^2}},\]
 we have $g(x) =U \psi_a(x)$ for $|x|=1$, and therefore $ g \vert_{S^{n-1}} \in M(B^n)\vert_{S^{n-1}}.$ Thus the group $\Aut(B^n)$ and the M\"obius group $M(B^n)$ produce the same group of transformations of $S^{n-1}.$
\end{proof}

Although the restrictions of both groups onto $S^{n-1}$ are the same,  the group $\Aut(B^n)$ is more adjusted to our purposes. Indeed, the transformations from group $\Aut(B^n)$ are fractional linear, and therefore leave invariant the family of intersections of $B^n$ with affine $k$-planes, on which the Funk transforms are defined. At the same time, the group $M(B^n)$ consists of fractional-quadratic transformations and hence does not possess the above property.

\subsection{Extensions of automorphisms from lower dimensional balls}

Let $E_0 \in Gr_0(n,k).$ Then $B^k=E_0 \cap B^n$  is the  $k$-dimensional open unit ball in $E_0$.  We denote by  $B^k_{\bbc}$  the  $k$-dimensional unit ball in $\bbc^n$ the real part of which is $B^k$. If $\Aut(B^k_{\bbc})$  is the group of   biholomorphic automorphisms of  $B^k_{\bbc}$, the restriction of  $\Aut(B^k_{\bbc})$ onto $B^k$ will be denoted by $\Aut(B^k)$.

The following lemma is a real version of  \cite[2.2.8]{Rudin}.
\begin{lemma} \label{L:ext}
For any $\psi \in \Aut(B^k)$ there exists $\Psi \in \Aut(B^n)$ such that $\Psi\vert_{B^k}=\psi.$
\end{lemma}
\begin{proof} We proceed as in  \cite{Rudin}. Let $\psi=\varphi_a V$, $a \in B^k$, where $ V $ is an orthogonal transformation  of $E_0$.
The extension of $V$ is clear, so it suffices to construct the extension of $\varphi_{a}.$
Decompose $\mathbb R^n= E_0 \oplus E_0^{\perp}$, $x = x^{\prime}+ x^{\prime\prime}$, and define
$$\Psi(x^{\prime}+x^{\prime\prime})=\varphi_a (x^{\prime})- \frac{\sqrt{1-|a|^2}\, x''}{1- \langle x^{\prime}, a \rangle }.$$
Since $\langle x ^{\prime}, a \rangle=\langle x, a \rangle$, the  definition (\ref{E:phi-def}) yields that $\Psi$ is the needed extension of $\varphi_a$ onto $B^n.$
\end{proof}

\begin{lemma} \label{L:E-c} Let $E$ be an affine  $k$-plane meeting $B^n.$ Denote $c \in E \cap B^n$  the center of the $k$-ball $E \cap B^n.$
If $ g \in \Aut(B^n)$ satisfies $g(E \cap B^n) \subset E \cap B^n$ and $g(c)=c,$ then there exists
$U \in O(n)$ such that $g(x)=U(x)$ for $ x \in E \cap B^n.$
\end{lemma}
\begin{proof}
We set $\mathbb R^n =\mathbb R^k \times \mathbb R^{n-k}$ so that $ x=(x^{\prime}, x^{\prime\prime}).$
Using a suitable orthogonal transformation, we can assume that
$E= \{ x \in B^n: x^{\prime\prime}=c^{\prime\prime} \} . $

By Definition \ref{D:Aut}, there exists $\Psi \in \Aut(B^n_{\mathbb C})$  such that $\Psi(x)=g(x), \ x \in B^n.$
Represent $\Psi(z)=\big(\Psi^{\prime}(z), \Psi^{\prime\prime}(z) \big)  \in \mathbb C^k \times \mathbb C^{n-k}.$

The intersection $E \cap B^n$ is the $k$- ball with the center $c$ and radius $r=\sqrt{1-|c|^2},$ i.e. $E \cap B^n=B(c,r).$ By the condition
$\Psi^{\prime\prime}(x^{\prime}, c^{\prime\prime})=c^{\prime\prime},$ for all $ \ x^{\prime} \in B^k(c^{\prime}, r).$
The uniqueness theorem for holomorphic functions yields that  $\Psi^{\prime\prime}$ is constant in the corresponding complex ball :
$\Psi^{\prime\prime}(z^{\prime}, c^{\prime\prime})=c^{\prime\prime},$
for $ z^{\prime} \in B^k_{\mathbb C}(c^{\prime}, r).$

Thus, $\Psi$ preserves the section of the unit complex ball $B^n_{\mathbb C}$ by the complex plane $\z^{\prime\prime}=c^{\prime\prime}.$
It follows that $\Psi_{c^{\prime\prime}} : z^{\prime} \rightarrow \Psi^{\prime} (z^{\prime}, c^{\prime\prime})$
is a holomorphic mapping of the complex $k$-ball $B^k_{\mathbb C}( c^{\prime}, r)$ onto itself. Moreover, it preserves the center of the ball which follows from $\Psi(c)=g(c)=c.$ By Cartan theorem (\cite{Rudin}, Theorem 2.1.3; see also the proof of  Theorem 2.2.5 there) $\Psi_{c^{\prime\prime}}$ is a unitary transformation of $\mathbb C^k,$, i.e., $\Psi_{c^{\prime\prime}}(z^{\prime})=\Psi^{\prime} ( z^{\prime} , c^{\prime\prime})=U^{\prime}(z^{\prime})$ for some $U^{\prime} \in U(k).$

Now the desired  unitary transformation $U$ of $\mathbb C^n$ can be defined as
$U(z^{\prime}, z^{\prime\prime})=(U^{\prime}(z^{\prime}), z^{\prime\prime}).$ Indeed,
if $ x \in E \cap B^n, \ x=(x^{\prime}, c^{\prime\prime}) $ then by the construction $g(x)=g(x^{\prime}, c^{\prime\prime}) =U(x^{\prime}, c^{\prime\prime})=U(x).$
\end{proof}

\subsection{ Action of $\Aut(B^n)$ on affine Grassmanian}

The group $\Aut(B^n)$ acts on the affine Grassmanians (Lemma \ref{L:phi-prop}, (6)). If $E \in Gr(n,k)$ and $g \in \Aut(B^n)$ then $g(E \cap B^n)$ is a $k$-dimensional ball in $E$ and $g(E \cap S^n)$ is
a $(k-1)$- subsphere of $S^{n-1}.$ The following lemma specifies ( for $g=\varphi_a$) the center and radius of this subsphere.

\begin{lemma}\label{L:radius}
Let $E_0 \in Gr_0(n,k)$ and $ a \in B^n.$ Then $\varphi_a(E_0 \cap S^{n-1})$
is the $(k-1)$-dimensional sphere in $E_0$ with the center $c = \varphi_a(a^{\prime}) ),$
where $a^{\prime}= Pr_{E_0} a$ is the orthogonal projection of $a$ to $E_0.$ The radius of this sphere is
$$r =\sqrt{ \frac{1-|a|^2}{ 1- |a^{\prime}|^2} },$$

\end{lemma}
\begin{proof} Let $c$ be the center of the $k$-ball $\varphi_a(E_0 \cap B^n)$ and $x_0 = \varphi_a(c).$

Since $\varphi_a$ is an involution,  $x_0 \in E_0 \cap B^n$ and $c=\varphi_a(x_0).$
Since $c$ is the center of the disc $\varphi_a(E_0 \cap B^n),$ then for any
$ y \in \varphi_a(E_0 \cap B^n), \ y=\varphi_a(x),$ holds
$\langle y-c,  c \rangle =0, $
i.e. $|c|^2= \langle y, c \rangle$ or
$1-|c|^2=1-\langle y, c \rangle.$

Substitute $c=\varphi_a(x_0), \ y = \varphi_a(x):$
$$1-|\varphi_a (x_0)|^2=1- \langle \varphi_a (x), \varphi_a(x_0) \rangle. $$
This equality transforms, due to identity (\ref{three}) in Lemma \ref{L:phi-prop}, to
$$\frac{ 1 -\langle x_0, x \rangle}{1 - \langle x, a \rangle}=\frac{1-|x_0|^2}{1- \langle x_0, a \rangle}=C.$$ Here $ x \in E_0 \cap B^n$ is arbitrary.
Substituting $x=0$ yields $C=1$ and hence we obtain $\langle x,  x_0-a \rangle =0.$  This means that $x_0=Pr_{E_0} a=a^{\prime}$ and hence  $c=\varphi_a(x_0)=\varphi_a(a^{\prime}).$

Then the radius is $r=\sqrt{1-|c|^2}=\sqrt{1-|\varphi_a(a^{\prime})|^2 }$ and the expression for the radius $r$  follows again from the identity (\ref{three1}) in Lemma (\ref{L:phi-prop}).
\end{proof}

\section{Pullback measures. Jacobians}

Let $E_0 \in Gr_0( n,k).$ For any automorphism $g \in \Aut(B^n)$ the image
$g(E_0 \cap B^n)$ is an affine $k$-dimensional subset of $B^n,$ i.e.,
$$g(E_0 \cap B^n)=E \cap B^n,$$
where  $E$ is an affine $k$-plane.

Denote $dA_{E_0}, \ dA_{E}$ the surface area Lebesgue measures on the $(k-1)$-dimensional spheres
$E_0 \cap S^{n-1} , \ E \cap S^{n-1},$ correspondingly.

The pullback measure
$$ g^*(dA_{E})=dA_{E} \circ g,$$
obtained from $dA_E(y)$ by the change of variables $y=g(x),$
is defined on the $(k-1)$-dimensional subsphere $ E_0 \cap S^{n-1} \subset S^{n-1}$ and is absolutely continuous with respect to the surface area measure on this subsphere:
$$(dA_{E} \circ g) (x) =J_{g,E_0}(x) dA_{E_0}(x) .$$

Our aim is to obtain an explicit expression for the Radon-Nikodym derivative (Jacobian)
$$J_{g,E_0}= \frac{ dA_{E} \circ g } { dA_{E_{0}}}.$$

We begin with the formula of change of variables in the surface area measure on $S^{n-1}:$

\begin{lemma}\label{L:dAn} \cite [Lemma 2.2]{AgR} Let  $a \in B^n.$ Then
\begin{equation}\label{E:dAn}
(dA_{S^{n-1}} \circ \varphi_a) (x) = \Bigg ( \frac{\sqrt{1-|a|^2} }{1- \langle a,  x \rangle}  \Bigg) ^{n-1} dA_{S^{n-1}}(x).
\end{equation}
and
\begin{equation} \label{E:full}
\int\limits_{S^{n-1}}\Bigg ( \frac{\sqrt{1-|a|^2} }{1- \langle a,  x \rangle}  \Bigg) ^{n-1} dA_{S^{n-1}}(x) = \sigma_{n-1},
\end{equation}
where $\sigma_{n-1}$ stands for the surface area of $S^{n-1}.$
\end{lemma}

Formula (\ref{E:full}) follows from (\ref{E:dAn}) by taking integrals over $S^{n-1}$ of both parts in (\ref{E:dAn}) and the change of variables $y=\varphi_a(x)$ in the integral in the left hand side.

\begin{remark} The right hand side in (\ref{E:dAn}) represents the Poisson measure  with respect to $\Aut(B^n)$-invariant Laplacian in the ball $B^n$
(see (\cite{Rudin}, 3.3.1 (1));  (\cite{Stoll}, (5.18)).
\end{remark}

Now we turn to measures, invariant under isotropy subgroups of $\Aut(B^n),$ i.e., subgroups with a fixed point.

Let $a \in B^n.$ Denote $\Aut_a(B^n)$ the isotropy group (stabilizer) of the point $a:$
$$\Aut_a(B^n)=\{g \in \Aut(B^n): g(a)=a\}. $$

\begin{lemma}\label{L:uniq}
Let $a \in B^n.$  Let $d\mu$ be a measure on $S^{n-1}.$ Then $d\mu$ is $\Aut_a(B^n)$-invariant if and only if
it has the form $C \ dA_{S^{n-1}} \circ \varphi_a,$
for some constant $C.$
\end{lemma}

\begin{proof}
Since $\varphi_a(0)=a, $ the stabilizers  $\Aut_a(B^n)$ and $\Aut_0(B^n)=0(n)$  are conjugate to each other by means the involution $\varphi_a = \varphi_a^{-1}:$
$$\Aut_a (B^n)=\varphi_a^{-1}  O(n) \varphi_a, \ O(n)=\varphi_a^{-1} \Aut_a(B^n) \varphi_a .$$
Then the automorphism $\varphi_a$ induces  an isomorphism between the invariant measures with respect to the two subgroups. Namely,
the measure $d\mu$ on $S^{n-1}$ is $\Aut_a(B^n)$-invariant if and only if the measure $d\mu \circ\varphi_a$ is $O(n)$-invariant, and therefore has the form  $d\mu \circ \varphi_a = C \ dA_{S^{n-1}}, \ C = \const.$ This is equivalent to $d\mu = C \ dA_{S^{n-1}} \circ \varphi_a.$

\end{proof}

\begin{corollary} \label{C:dAk} Let $E_0 \in Gr_0(n,k), \ B^k = E_0 \cap B^n, v \in B^k.$
Let $\Aut_{E_0}(B^n) = \{g \in \Aut(B^n): g(E_0 \cap B^n) \subset (E_0 \cap B^n)\}$ and $\Aut_{E_0, v}= \{ g \in \Aut(B^n): g(v)=v \}. $
Then any $\Aut_{E_0, v}(B^n)$-invariant measure $d\mu$ on $S^{k-1} = E_0 \cap S^{n-1}$
has the form
\begin{equation} \label{E:dAk}
d\mu(x)= C \ \Bigg ( \frac{\sqrt{1-|v|^2}}{1 - \langle x, v \rangle } \Bigg) ^{k-1} dA_{E_0}(x), \ C = \const.
\end{equation}
\end{corollary}

\begin{proof}
By Lemma \ref{L:ext}, every automorphism $\psi \in \Aut(B^k)$ extends to some $\Psi \in \Aut(B^n).$ If $\psi$ preserves $v,$ then $\Psi \in \Aut_{E_0, v}(B^n).$
It follows that if $d\mu$ is $\Aut_{E_0, v}$-invariant on $S^{k-1},$ then it is $\Aut(B^k)$-invariant on $S^{k-1}.$
Hence, by the $k$-dimensional version of Lemma \ref{L:uniq} (with $a = v$) we conclude that $d\mu= C \ dA_{E_0}\circ \varphi_v.$
Finally, the $k$-dimensional version of Lemma \ref{L:dAn} implies the explicit expression for $d\mu.$
\end{proof}

The following Lemma presents formula of change of variables in measures on affine cross-sections and gives a key for constructing intertwining operators between shifted  Funk transforms.
\begin{lemma} \label{L:pullback}
Let $E_0 \in Gr_0(n,k)$ and $E \in Gr(n,k)$ be such that $ E \cap B^n = g(E_0 \cap B^n), \ g \in \Aut(B^n).$ Then the pullback measure is
\begin{equation} \label{E:JJ}
(dA_{E} \circ g)(x) =J_a \big( (\varphi_a \circ g ) (x) \big) \big) dA_{E_0}(x),
\end{equation}
where $a=g(0)$ and
\begin{equation} \label{E:J}
J_a (y)=\Bigg ( \frac{\sqrt{1-|a|^2} }{ 1- \langle y, a \rangle} \Bigg )^{k-1},
\end{equation}
$a=g(0).$
In particular, if $g=\varphi_a$ then $\varphi_a \circ g= id$ and we have
\begin{equation}\label{E:dAEphi}
dA_{E} \circ \varphi_a (x) = J_a(x) dA_{E_0}(x).
\end{equation}
\end{lemma}

\begin{remark} The remarkable fact is that the Jacobian $J_a$ in (\ref{E:JJ}), (\ref{E:dAEphi})  depends only on the point $a=g(0)$ ( in the terminology in (\cite{Pala}, Section 3.1) the automorphisms from $\Aut(B^n)$ are factorable). This circumstance makes possible constructing  intertwining operators between the Funk transforms with different centers. It will be done in the next sections.
\end{remark}

\begin{proof}

The proof relies on the following arguments: 1) area surface measures on spheres are determined by the property of invariance  with respect to the isotropy groups of the centers,  2) automorphisms $g \in \Aut(B^n)$ transform isotropy groups to isotropy groups,  3) the associated pullback mappings transform invariant measures to invariant measures.  The explicit expressions for such measures follow from Corollary \ref{C:dAk}.

Represent the automorphism $g$ as $g=\varphi_a \circ V, \ V \in O(n), \ a = g(0).$
Then by the chain rule
$$(dA_E \circ g )(x)=J_{\varphi_a, E_0^{\prime}} (Vx) J_{V, E_0}(x) dA_{E_0}(x), $$
where we have denoted
$$E_0^{\prime}=V(E_0).$$

The orthogonal transformation $V$ preserves Lebesgue measure, hence $J_{V, E_0}(x)=1$ and
$$ (dA_{E_0^{\prime}})(Vx)=dA_{E_0}(x).$$
Therefore, the above equality transforms to
$$(dA_E \circ g)(x) = J_{\varphi_a, E_0^{\prime}} (Vx) dA_{E_0^{\prime}}(Vx).$$
Replacing here $y=Vx$ and $g(x)=\varphi_a(Vx)$ leads to
the  equivalent pullback relation
\begin{equation}\label{E:14}
dA_E \circ \varphi_a (y)=J_{\varphi_a,E_0^{\prime}}(y)dA_{E_0^{\prime}}(y).
\end{equation}
Thus, it suffices to compute Jacobian $J_{\varphi_a, E_0^{\prime} } (y), \ y \in E_0^{\prime} \cap B^n$ in (\ref{E:14}).
Then formula (\ref{E:JJ}) follows, because $y=Vx=\varphi_a^{-1} \big (g(x) \big )=\varphi_a \big (g(x) \big ).$

The set $E \cap B^n=\varphi_a(E_0^{\prime} \cap B^n)$ is a ball in the $k$-dimensional affine set $E.$  It can be identified with a ball $E=B^k(c,r)$ in $\mathbb R^k,$ with the center $ c \in E$ and the radius $r=\sqrt{1-|c|^2}$ as in Lemma \ref{L:radius}.

Define the pullback measure on $E_0^{\prime} \cap S^{n-1}:$
$$d\mu =dA_E \circ \varphi_a.$$
Let us prove that $d\mu$ is invariant with respect to any automorphism $g \in \Aut(B^k)$ of the $k$-ball $B^k = \Aut(E_0^{\prime} \cap B^n),$
having the point $\varphi_a (c)$ fixed. To this end, consider the automorphism
$$g_1=\varphi_a \circ g \circ \varphi_a.$$
Since $ g (E_0^{\prime} \cap B^n )=E_0^{\prime} \cap B^n$ then $g_1(E \cap B^n)= E \cap B^n.$
Also $g_1(c)=\varphi_a \big (g(\varphi_a(c)) \big)=\varphi_a(\varphi_a(c))=c.$
Thus,  $g_1$ maps the ball $E \cap B^n$ onto itself and preserves its center.

By Lemma \ref{L:E-c} there exists $U \in O(n)$ such that $g_1\vert_{E \cap B^n}= U\vert_{E \cap B^n}.$
Since Lebesgue measure $dA_E$ in invariant under rotations of the sphere $E \cap S^{n-1},$ we have $dA_E \circ g_1=dA_E \circ U=dA_E.$
This reads
$$ (dA_E \circ \varphi_a) \circ g \circ \varphi_a=dA_E$$

or
$$d\mu \circ g= dA_E \circ \varphi_a=d\mu.$$

By Lemma \ref{L:radius} $\varphi_a(c)=a^{\prime},$ where $a^{\prime}$  is the orthogonal projection of $a$ to $E_0^{\prime}.$
Applying Corollary \ref{C:dAk} to the linear space $E_0^{\prime},$ the measure $d\mu$ and the fixed point $v = \varphi_a(c) = a^{\prime}$ yields:
\begin{equation} \label{E:12}
(dA_{E}  \circ \varphi_a)(y) = C \ \Bigg( \frac{\sqrt{1- a^{\prime}|^2}}{ 1- \langle y, a^{\prime} \rangle} \Bigg )^{k-1} dA_{E_0^{\prime}}(y).
\end{equation}

The constant $C$ can be found by integration both sides of (\ref{E:12}) and comparing the full measures.
\begin{equation}\label{E:C}
\int\limits_{E_0^{\prime} \cap S^{n-1}} (dA_E \circ \varphi_a) (y))
= C \int\limits_{E_0^{\prime} \cap S^{n-1}}  \Bigg( \frac{\sqrt{1-|a^{\prime}|^2} }{ 1- \langle y, a^{\prime}\rangle  } \Bigg )^{k-1} dA_{E_0^{\prime}}(y).
\end{equation}

The left hand side reduces, by the change of variables $u=\varphi_a y,$  to
$$\int\limits_{E \cap S^{n-1}} dA_E(u)= \sigma_{k-1} r^{k},$$
where $\sigma_{k-1}$ is the surface area of the sphere $S^{k-1}.$

On the other hand,  formula (\ref{E:full}) with $n=k, \ S^{k-1}= E_0^{\prime} \cap S^{n-1}$ and $ a=a^{\prime}, $  yields that the right hand side of (\ref{E:C}) equals $C \sigma_{k-1}.$

It follows that $C= r^{k-1}$ and hence $C=\Big( \frac{\sqrt{1-|a|^2}}{\sqrt{1-|a^{\prime}|^2} } \Big )^{k-1}$ by Lemma \ref{L:radius}. Substituting the expression for $C$ in (\ref{E:C}) leads to the desired formula, because $\langle y, a ^{\prime} \rangle = \langle y,  a \rangle$
for all $ y \in E_0^{\prime}.$
\end{proof}

Lemma \ref{L:pullback} establishes the link, via the action on $B^n$ of the group $\Aut(B^n),$ between the surface area measures on affine and linear cross-sections of $S^{n-1}.$
The next lemma establishes a similar link for a case of two affine $k$-planes.
\begin{lemma}\label{L:pull1} Let $E_0 \in Gr_0(n,k), b \in E_0 \cap B^n.$  Consider a parallel affine $k$-plane
$E_1= E_0 +e,$  where $ e \in E_0^{\perp} \cap B^n.$
Denote $E$ the affine  $k$-plane such that $\varphi_b (E_1 \cap B^n)=E \cap B^n.$
Then the surface area measures $dA_{E}$ and $dA_{E_1}$ are related by
\begin{equation}\label{E:pull2}
(dA_{E} \circ \varphi_b)(y)=J_b(y)dA_{E_1}(y),
\end{equation}
where
\begin{equation}\label{E:Jacobian}
J_b(y)=\Bigg( \frac{\sqrt{1-|b|^2}}{ 1- \langle y, b \rangle} \Bigg)^{k-1}.
\end{equation}
\end{lemma}

\begin{proof}
The idea of the proof is to establish,using suitable automorphisms and Lemma \ref{L:pullback}, pullback relations of the measures $dA_{E}$ and $dA_{E_1}$ with $dA_{E_0}$ and then eliminate $dA_{E_0}$ from the two equalities.

Consider the automorphism $\varphi_e.$ If $ x \in E_0 \cap B^n$ then  $ \langle x, e \rangle  =0$ and from Definition \ref{E:phi-def}
$\varphi_a (x)= e-\sqrt{1-|a|^2} x.$  Therefore, $\varphi_e (E_0 \cap B^n)=E_1 \cap B^n$ is  the $k$-ball in $E_1,$ centered at $e,$ of radius $\sqrt{1-|e|^2}.$

Thus, we have the two consequent mappings
$$\varphi_e: E_0 \cap B^n  \rightarrow  E_1 \cap B^n , \ {\varphi_b}: E_1 \cap B^n  \rightarrow E \cap B^n.$$
Consider their composition:

$$g=\varphi_b \circ \varphi_e : E_0 \cap B^n \rightarrow E \cap B^n.$$

Denote $$c=g(0)=\varphi_b(e).$$
Lemma \ref{L:pullback}, applied to $g$ and $\varphi_e,$ implies the following pullback relations:
\begin{equation}\label{E:pullbacks}
\begin{aligned}
&(dA_{E}\circ g) (x) = \Bigg ( \frac{ \sqrt{1-|c|^2} }{ 1 - \big \langle \varphi_c \big( g(x) \big), c \big \rangle} \Bigg)^{k-1} dA_{E_0}(x), \\
&(dA_{E_1}\circ \varphi_e)(x)=\Bigg( \frac{ \sqrt{1- |e|^2} }{1 - \langle x, e \rangle} \Bigg)^{k-1}dA_{E_0}(x)=(\sqrt{1-|e|^2})^{k-1}dA_{E_0}(x).
\end{aligned}
\end{equation}
In the last equality we have used that $x \in E_0$ and $ e \perp E_0.$

The condition $\langle b, e \rangle =0$ and Lemma \ref{L:phi-prop}(\ref{three1}) imply that the enumerator in the first equality in (\ref{E:pullbacks}) is
$$\sqrt{1-|c|^2}=\sqrt{1-|\varphi_b (e)|^2}=\sqrt{(1-|b|^2)(1-|e|^2)}.$$

Therefore, by eliminating $dA_{E_0}$ from the two equalities, we obtain
\begin{equation}\label{E:11}
(dA_E \circ g)(x)=\Bigg( \frac{\sqrt{1-|b|^2}}{ 1- \big \langle \varphi_c \big (g(x) \big ), c \big \rangle } \Bigg)^{k-1}(dA_{E_1} \circ \varphi_e)(x).
\end{equation}
Denote
$y =\varphi_e(x).$  Then $x=\varphi_e(y)$ and $ g(x)=\varphi_b (\varphi_e(x))=\varphi_b(y).$
Also, when $x \in E_0 \cap S^{n-1}$ then $ y \in E_1 \cap S^{n-1}.$ Then (\ref{E:11}) can be rewritten as:
\begin{equation}\label{E:1}
(dA_E \circ \varphi_b)(y)=
\Bigg ( \frac{\sqrt{1-|b|^2}} { 1- \big \langle \varphi_{c} (\varphi_b(y)) , c \big \rangle } \Bigg )^{k-1} dA_{E_1}(y),
\end{equation}
$c=\varphi_b(e), \ y \in E_1 \cap S^{n-1}.$

To transform the obtained equality to the required form, we have to show that
$ 1- \big \langle \varphi_{c} (\varphi_b(y)) , c \big \rangle =1-\langle b, y \rangle.$ It follows by
using repeatedly the identities from Lemma \ref{L:phi-prop}(iii), and also the relations $\langle b, e \rangle = \langle y - e , e \rangle = 0.$


\end{proof}

\section{ Intertwining operators between the Funk transforms with different centers}

\subsection{Two standard Funk transforms}
 Our nearest aim is to relate the Funk transforms $F_a, \ |a|< 1$ and $F_b, \ |b|>1$ with two {\it standard Funk transforms}, $F_0$ and $\Pi_b,$ respectively, which are defined as follows.

The first one is just the Funk transform centered at the origin, i.e., the transform $F_a,$ defined in (\ref{E:Fa-def}), with $a=0.$
It is defined on the linear Grassmanian $Gr_0(n,k).$

The second one, denoted by $\Pi_b, b \neq 0,$ is  called the {\it parallel slice transform} \cite{AgR1}.
It can be formally obtained  from $F_b (E)$ when the center $b$ tends to infinity, so that the plane $E$ through $b$ becomes parallel to the direction $b.$

Specifically, let $Gr^b(n, k)$ be the submanifold of $Gr(n, k)$ of all $k$-planes $E$ meeting $B^n$ and parallel to the vector $b,$ i.e., having the form $E = e + E_0, $ where $E_0 \in Gr_0(n,k),  \ b \in E_0$ and $|e| \leq 1.$

 We define
\begin{equation}\label{E:Pib}
\Pi_b f )(E)=\int\limits_{E \cap S^{n-1}} f(x)dA_E(x),   E \in Gr^b(n,k).
\end{equation}

\subsection{The interior center $a$. Link between the transforms $F_a$ and $F_0$}

Fix $a \in \mathbb R^n, |a|<1.$
Let $E \in Gr_a(n,k)$ be an affine $k$-plane, containing $a.$
The automorphism $\varphi_a$ maps $E \cap B^n$ onto an affine $k$-section $E_0 \cap B^n=\varphi_a(E \cap B^n)$ and since $\varphi_a(a)=0,$
we have $E_0 \in Gr_0(n,k).$ In turn, $E \cap B^n =\varphi_a(E_0 \cap B^n).$

By Lemma \ref{L:pullback}, (\ref{E:J}), for any $f \in C(S^{n-1})$ holds:
\begin{equation}\label{E:int=int}
\int\limits_{E \cap S^{n-1}} f(y)dA_{E}(y)=\int\limits_{E_0 \cap S^{n-1}} f\big(\varphi_a(x)\big)J_a(x)dA_{E_0}(x),
\end{equation}
where
$$J_a(x)= \Bigg ( \frac{\sqrt{1-|a|^2}}{1- \langle x, a \rangle} \Bigg) ^{k-1}.$$

Denote
\begin{equation}\label{E:Ma}
(M_a f)(x)= f \big(\varphi_a (x) \big)  J_a(x).
\end{equation}
Then  ( {\ref{E:int=int}) can be written as $(F_a f)(E)=(F_0 M_a f)(E_0).$
Thus, the operator $M_a: C(S^{n-1}) \rightarrow C(S^{n-1})$ serves intertwining operator between
the shifted Funk transform $F_a$ and the transform $F_0$ centered at the origin. More specifically,
\begin{equation}\label{E:intertwin}
\begin{aligned}
&(F_a f) (E)= \big( F_0 M_a f\big) (\varphi_a(E)), \ E \in Gr_a(n,k), \\
&(F_a f)(\varphi_a(E_0))=\big( F_0 M_a f\big) (E_0), \ E_0 \in Gr(n,k).
\end{aligned}
\end{equation}
These relations were proved in \cite{AgR} using different methods.

\subsection{The exterior center $b$. Link  between the transforms $F_b$ and $\Pi_b$}

In this section, we obtain analogues of intertwining relations (\ref{E:intertwin})  for the case of exterior center.
We will show that if $b|>1$ then the transform $F_b$ is linked to the parallel slice transform
$\Pi_b.$ The intertwining operator is obtained by means of the automorphism $\varphi_{b^*},$ where $b^*$ is the inversion $ b^*=\frac{b}{|b|^2}$ with respect to the sphere $S^{n-1}.$

We start with the following simple fact.
\begin{lemma}\label{L:parallel}
An affine $k$-plane $E$ contains the point $b, |b|>1,$ if and only if  the $k$-plane $E_1=\varphi_{b^*}(E)$ is parallel to the vector $b.$
\end{lemma}

\begin{proof}
Since affine $k$-planes parallel to $b$ are unions of affine lines  parallel to $b,$  it suffices to prove the statement for lines.

Let $L=\{b +\lambda v , \ \lambda \in \mathbb R \}$ be an affine line containing the point $b.$
Then from (\ref{E:phi-def})
$$
\begin{aligned}
&x=\varphi_{b^*}(b+\lambda v)=\frac{ b^* -P_{b_*}(b+\lambda v)-\sqrt{1-|b^*|^2} Q_{b^*}( b+ \lambda v)}{ - \lambda \langle b^*, v \rangle } \\
&=\frac{1}{\lambda} \frac{ b-b^*}{ \langle b^*, v \rangle } + \frac{ P_{b^*} v + \sqrt{1-|b^*|^2} Q_{b^*} v}{\langle b^*, v \rangle},
\end{aligned}
$$
because $ P_{b^*}b=b, \ Q_{b^*} b=0, \ \langle b^*, b \rangle =1.$

Since $b^*$ is proportional to $b,$ we conclude that any $ x \in \varphi_{b^*}(L)$ has the form $x= t b + d, \ t \in \mathbb R,$ where $d$ is a fixed vector. Therefore the affine line
$\varphi_{b^*}(L)$ is parallel to the vector $b.$

Conversely, suppose that $L_1=\varphi_{b^*}( L) $ is parallel to $b,$  i.e.,
$$L_1=d+ \mathbb R \cdot b.$$
Since $L=\varphi_{b^*} (L_1),$ every $y \in L$ has the form $ y=\varphi_{b^*}( d+ \lambda b), \ \lambda \in \mathbb R:$
$$y=\frac{ b^* - P_{b^*}(d+\lambda b) - \sqrt{1-|b^*|^2} Q_{b^*}(d+ \lambda b)} {1-\langle b^*, d+ \lambda b \rangle}.$$
From here,
$$\lim\limits_{\lambda \to \infty}y= \frac{P_{b^*}b+\sqrt{1-|b^*|^2}Q_{b^*}b}{ \langle b^*, b \rangle }=b.$$
Therefore, $b \in L,$ because $L$ is a  closed set. Lemma is proved.
\end{proof}

Thus, the automorphisms $\varphi_{b^*}$ interchanges the two Grassmanians:
$$\varphi_{b^*}: Gr^b(n,k) \to Gr_b(n,k), \ Gr_b(n,k) \to Gr^b(n,k).$$
It remains to establish the relation between the transforms $F_b,$ defined on $Gr_b(n,k)$ and $\Pi_b,$
defined on $Gr^b(n,k).$

Let  $E \in Gr_b(n,k).$ Then $E_1=\varphi_{b^*}(E) \in Gr^b(n,k)$ and $E_a=E_0+ e$ for some vector $e$ and linear subspace $E_0 \in Gr_0(n,k).$
Lemma \ref{L:pullback1} with $b^*$ in place of $b$ and change of variables $u=\varphi_{b^*}(y)$ imply:
$$
\begin{aligned}
&(F_b f)(E)=\int\limits_{E \cap S^{n-1}} f(u) dA_{E}(u)=\int\limits_{E^1 \cap S^{n-1}} f(\varphi_{b^*}(y))J_{b^*}(y) dA_{E_1}(y) \\
&=F \big ( ( f\circ \varphi_{b^*}) J_{b^*} \big)(E_1) ,
\end{aligned}
$$
where the expression for the Jacobian $J_{b^*}$ is given by (\ref{E:Jacobian}), with $b^*$ in place of $a.$

Thus, the analogues of relations (\ref{E:intertwin}) for the case of the exterior center look as follows:
\begin{equation}\label{E:pullback1}
\begin{aligned}
&(F_b f)\big (\varphi_{b^*}(E_1) \big)=(\Pi_b M_{b^*} f)(E_1), \ E_1 \in Gr^b(n,k), \\
&(F_b f)(E)= (\Pi_b M_{b^*} f) (\varphi_{b^*}(E)), \ E \in Gr_b(n,k ).
\end{aligned}
\end{equation}
where
\begin{equation}\label{E:intertwin1}
M_{b^*}f(x)= f \big(\varphi_{b^*}(x)\big)\Bigg( \frac{ \sqrt{1-|b^*|^2} } {1-\langle x, b^* \rangle} \Bigg)^{k-1} ,
\end{equation}
in accordance with definition (\ref{E:Ma}).

\section{Kernel of  single shifted Funk transforms}

The intertwining relations enable us to characterize the kernels of transforms $F_a , |a|<1,$ and  $ F_b , \ |b|>1,$  using the similar results for the standard transforms $F_0$ and $\Pi_b.$ In order not to overload notations, we will sometimes omit parenthesis in $\tau_a(x), \ \varphi_a(x),$ where it does not lead to confusion.

\subsection{Kernels of the standard Funk transforms}

Recall that by standard Funk transform we understand the transforms of two types: the classical Funk transform $F_0$ with the center at $0$  and the parallel slice Funk transform $\Pi_b$ which corresponds to the center at $\infty.$

It is well known (e.g., \cite{H11}), and it can be easily proved using Fourier decomposition in spherical harmonics,
that the kernel of the transform $F_0$ consists of odd functions on $S^{n-1}:$
$$
\ker F_0= \{f \in C(S^{n-1}): f(-x)=-f(x) \}.
$$

The kernel of the transform $\Pi_b,$ too, is described in terms of an oddness condition, but in this case, with respect to
the reflection
$$\sigma_b: x -2\frac { \langle x, b\rangle}{|b|^2} b $$
across the hyperplane $\{ \langle x, b \rangle =0\}.$
In geometric terms,  $\sigma_b x, \ x \in S^{n-1},$ is a point $x^{\prime} \in S^{n-1},$  such that the vector $ x^{\prime}-x $ is proportional to the vector $b.$
\begin{lemma} \label{L:kernels}
The kernel of the parallel slice transform $\Pi_b$ consists of all $\sigma_b$-odd functions:
$$\ker \Pi_b =\{f \in C(S^n): f \circ \sigma_b=-f \}.$$
\end{lemma}

\begin{proof} A proof using  relationship between $\Pi_b$ and the $k$-plane Radon-John transform can be found in [ \cite{AgR1}, Theorem 3.2]. Below we present an alternative analytic proof not assuming any knowledge of the Radon transform theory.

It is clear that if $f \in C(S^{n-1}$ is $\sigma_b$-odd then $ f \in \ker \Pi_b.$
Conversely, let us show any $f \in \ker \Pi_b$  is $\sigma_a$-odd. Decompose $f$ into the sum of $\sigma_b$-even and $\sigma_b$-odd functions:
$$f =f^+ + f^{-},$$
where
$$f^{\pm}=\frac{1}{2}( f \pm f \circ \sigma_b).$$
Since $f^{-}$ is odd with respect to the reflection around the plane $\langle x, b \rangle =0,$ we have $\Pi_b f^{-}=0$ and hence $\Pi_b f^+ = \Pi_b f=0.$ We want to prove that $f^{+}=0.$

Consider the distribution $\Psi=f^{+}\delta_{S^n}$ where $\delta_ {S^n}$ is the delta-function on the unit sphere.
The Fourier transform
$$\widehat \Psi(y) = \int\limits_{ S^{n-1}} e^{ -i \langle y, x \rangle } f^{+}(x)dA_{S^{n-1}}(x)$$
 is an eigenfunction of the Laplace operator in $\mathbb R^n:$
 $$(\Delta +1) \widehat \Psi= 0.$$
Fix $y \in b^{\perp}.$   Write $y=r\omega_1, \  |\omega_1|=1, \ r \geq 0.$ Add unit vectors $\omega_j \in b^{\perp}, \ j=2, n-1$
so that the system $\omega_1, ... \omega_{n-1}$ form an orthonormal  basis in the linear $(n-1)$-space $b^{\perp}.$

Represent the integral for $\widehat \Psi (y)$  as the iterated one:
$$\widehat \Psi(y)=\widehat \Psi( r\omega_1) =
\int\limits_{B^{n-k}} e^{-i r t_1} \int\limits_{S^{n-1} \cap E_t} f^+(x) dA_{k-1} (x) d\nu(t),$$
where $ t=(t_1, ..., t_{n-k}) \in B^{n-k},$ \
$E_t= \{ x \in \mathbb R^n:  \langle x, \omega_j \rangle = t_j, \ j = 1,..., n-k \}$ and
$d\nu(t)$ is a certain measure on the $n-k$-dimensional  unit ball $ B^{n-k}.$

The $k$-planes $E(t)$ are parallel to $b$ and hence  the condition  $\Pi_b f^+=0$ implies that the integrals over the $k$-planes $E(t)$ are zero. Therefore $\widehat \Psi(y)=0.$ The vector $y \in b^{\perp}$ ia arbitrary, so $\widehat \Psi$ vanishes on $b^{\perp}.$  Since $f^+$ is $\sigma_b$-even, $\widehat \Psi$ is $\sigma_b$-even as well and hence also the normal derivative  $\partial_b \widehat \Psi(y)=0,$ for all $y \in b^{\perp}.$

Thus, the function $\widehat \Psi$ satisfies on the hyperplane
$b^{\perp}$ zero Dirichlet and Neumann conditions, which  form a full set of Cauchy data for the operator $\Delta+1.$ Due to the uniqueness of the solution, which follows, for example, from Cauchy-Kowalevski theorem (cf. \cite{Folland}, (1. 25)) we have $\widehat \Psi=0$ . Then $\Psi=0$ and $ f^+=0, \ f=f^-.$
\end{proof}

Now we are passing to describing the kernels of the shifted Funk transforms.
\subsection{ Kernel of $F_a \ (|a|<1) $ }

Let $|a|<1$ and suppose that $(F_a f)(E)  =0$ for some $f \in C(S^n)$ and all $E \in Gr_a(n,k).$
By (\ref{E:intertwin}), it is equivalent to
$$(F_0  M_a f) (E_0) = F_0 \big ( (f \circ \varphi_a) J_a \big )(E_0) =0, \ E_0 \in Gr_0(n,k).$$
Thus, $ f \in \ker F_a$ if and only if  $M_a f \in \ker F_0,$ that is,
$$f \big (\varphi_a (x) \big)J_a(x) =- f\big(\varphi_a (-x)\big) J_a(-x).$$
Replace in the equality $x$ by $\varphi_a x$ (in order not to overload notations, sometimes we will omit parenthesis and write $\varphi_a x$) :
$$f(x)J_a(\varphi_a x) = -f (\tau_a x)J_a(-\varphi_a x),$$
where
\begin{equation} \label{E:taua-def}
\tau_a x = \varphi_a \big (-\varphi_a (x) \big ).
\end{equation}
Thus, the kernel $\ker F_a$ consists of all function $f \in C(S^n)$ satisfying
\begin{equation}\label{E:rhoa-def}
 f(x) = - \rho_a(x)f(\tau_a x), \ \rho_a(x) = \frac{J_a(-\varphi_a x)}{J_a(\varphi_a x) }.
\end{equation}

The relation $\varphi_a \tau_a x=-\varphi_a x$ implies:
\begin{equation}\label{E:rho-tau}
\rho_a(\tau_a x)\rho_a(x)=1.
\end{equation}

The mapping $\tau_a$ coincides with that in Corollary \ref{C:finite} and has a clear geometric meaning:

\begin{lemma} \label{L:tau} The mapping $\tau_a : S^n \to S^n,$ defined by  (\ref{E:taua-def}), is a symmetry of the unit sphere $S^n$ with respect to the point $a$ ( $a$-symmetry).  More specifically,
$\tau_a x =x^{\prime}, $ where $ x^{\prime} \in S^{n-1},$ is the second point of the intersection $S^{n-1} \cap L_{x,a}$ of the unit sphere with the straight line $L_{x,a}$ joining $x$ and $a.$ The analytic expression is
\begin{equation}\label{E:tau-analytic}
 \tau_a x = x+ 2 \frac{1 - \langle x, a \rangle }{|a-x|^2}(a-x).
\end{equation}
\end{lemma}

\begin{proof} The automorphism $\varphi_a$ preserves affine subsets (Lemma \ref{L:phi-prop}, (6)),
therefore the segment $[x, x^{\prime}], \ x \in S^{n-1},$
 is mapped to the segment $[\varphi_a (x), \varphi_a (x^{\prime}) ].$ Since $[x,x^{\prime} ]$  contains $a$, the image $\varphi_a \big( [x,x^{\prime}] \big)$
contains $\varphi_a(a) = 0.$   Therefore, the end points $\varphi_a (x)$ and $\varphi_a (x^{\prime})$ belong to $S^{n-1}$ and are symmetric with respect to $0:$
$\varphi_a(x^{\prime}) = -\varphi_a (x).$
Applying $\varphi_a$ to the both sides leads to $x^{\prime}=\varphi_a(-\varphi_a x)=\tau_a x.$  The obtained geometric description of $\tau_a x$ immediately implies its analytic expression (\ref{E:tau-analytic}).
\end{proof}
\begin{remark} \label{E:tau-extend} The geometric definition of $\tau_a$  in Lemma \ref{L:tau} and  formula (\ref{E:tau-analytic}) literally extend to the case $|a|>1.$
\end{remark}
\begin{lemma}\label{L:rhoa} The function $\rho_a, \ |a|<1, $ in (\ref{E:rhoa-def}) equals
\begin{equation}\label{E:rho-formula}
\rho_a (x)=\Big ( \frac{1-|a|^2}{|x-a|^2} \Big)^{k-1}.
\end{equation}
\end{lemma}

\begin{proof}
We have
$$
\begin{aligned}
&J_a(-\varphi_a x)=\Big ( \frac { \sqrt{1-|a|^2} } { 1  + \langle \varphi_a x, a \rangle } \Big)^{k-1},\\
&J_a(\varphi_a x)= \Big ( \frac {  \sqrt{1-|a|^2} } { 1  - \langle \varphi_a x, a \rangle}\Big)^{k-1}
\end{aligned}
$$
and hence
$$\rho_a(x)=\Big  (  \frac {1- \langle \varphi_a x, a \rangle}{1+ \langle \varphi_a x, a \rangle} \Big) ^{k-1}.$$
Using Lemma \ref{L:phi-prop}, (\ref{three} ) with $y=0$ we can write:
$$
\begin{aligned}
&1- \langle  \varphi_a x, a \rangle =\frac{1-|a|^2}{1- \langle x, a \rangle} \\
&1+ \langle  \varphi_a x, a \rangle =\frac{1- 2 \langle x, a \rangle +|a|^2} {1- \langle x, a \rangle}.
\end{aligned}
$$
Since $|x|=1,$ the equality (\ref{E:rho-formula}) follows.
\end{proof}

The identity (\ref{E:rhoa-def}) together with Lemma \ref{L:rhoa} give the following description of the kernel of the Funk transform $F_a$ with interior center:
\begin{theorem}\label{T:kerFa} Let $|a|<1.$ The kernel $\ker F_a$ consists of all functions $f \in C(S^{n-1})$ satisfying the condition
$$f (x)=-\rho_a(x) f(\tau_a x),$$ where the mapping $\tau_a(x)$ is defined in (\ref{E:taua-def}) and Lemma \ref{L:tau},  and the weight function $\rho_a(x)$ is given by (\ref{E:rho-formula}).
\end{theorem}

Now we will consider the case of exterior center.

\subsection{Kernel of $F_b \ (|b|>1)$ } \label{S:Fb}

Now we want to describe the kernel of the Funk transform $F_b$ with the exterior center $b.$
By (\ref{E:pullback1}), the associated standard Funk transform in this case is the parallel slice transform $\Pi_b. $  Using the description of $\ker \Pi_b$ in Lemma \ref{L:kernels} and the same arguments as in the previous section for the case of the interior center, we obtain:
\begin{equation}\label{E:kerFbnotfinal}
\ker F_b = \{ f \in C(S^{n-1}: f(x)=-\widetilde \rho_{b^*}(x) f (\widetilde \tau_{b^*}x) \},
\end{equation}
where the mapping $\widetilde \tau_{b^*}$ and the weight $\widetilde \rho_{b^*}$ are associated with
the symmetry  $\sigma_{b^*}$ with respect to the hyperplane $\langle x , b^* \rangle = 0$ in a similar way as $\tau_a$ and $\rho_a$  are associated with the symmetry $\tau_0: x \to -x$ with respect to the origin, i.e.,
\begin{equation}\label{E:taub*}
\widetilde \tau_{b^*}x=\varphi_{b^*} \big (\sigma_b  (\varphi_{b^*} x ) \big).
\end{equation}
and
\begin{equation}\label{E:rho*}
\widetilde \rho_{b^*}(x)=\frac { J_{b^*} \big ( \sigma_b  (\varphi_{b^*} x ) \big)  } { J_{b^*}  (\varphi_{b^*} x ) }.
\end{equation}
Now we will show that the mapping $\widetilde \tau_{b^*}$ and the weight function $\widetilde \rho_{b^*}$, respectively, are defined by the same formulas (\ref{E:tau-analytic}) and (\ref{E:rab}) as $\tau_b$ and $\rho_b$ , naturally extended for the case $|b|>1$ (see Remark \ref{E:tau-extend}).
\begin{lemma} \label{L:nub} For any $x \in S^{n-1},$ the point
$\widetilde \tau_{b^*}(x),$ defined by (\ref{E:taub*}), and the
 $b$-symmetric point $\tau_b(x)$ coincide.
\end{lemma}
\begin{proof}
Consider the segment $[x, \tau_b x].$ By the definition of $\tau_b x,$ this segment is a chord of $\overline B^n$ obtained by intersection with  the straight line $L_{x,b}$ joining $x$ and $b.$ By Lemma \ref{L:parallel}, the image $\varphi_{b^*}([x, \tau_b x])$  is a segment of an affine line, parallel to $b.$ Therefore, its end points, which belong to $S^{n-1},$ are $\sigma_b$-symmetric: $\sigma_b \big(\varphi_{b^*} (x) \big)=  \varphi_{b^*} (\tau_b x ).$
Since $\varphi_{b^*}$ is an involution, then
$\tau_b x= \varphi_{b^*} \big( \sigma_b ( \varphi_{b^*} (x)  ) \big)    =\widetilde \tau_{b^*}x.$
\end{proof}

It remains to compute the weight function which we have temporarily denoted $\widetilde \rho_{b^*}.$
\begin{lemma} \label{L:Rb} For all $ x \in S^n$ and $|b| >1,$
$$\widetilde \rho_{b^*}(x)= \rho_b(x):=\Big( \frac {|b|^2-1}{|x-b|^2} \Big)^{k-1}.$$
\end{lemma}

\begin{proof}
We have
$$J_{b^*} \big (\sigma_b (y) \big)=\Bigg ( \frac{\sqrt{1-|b^*|^2} }{ 1- \langle \sigma_b y, b^* \rangle   }\Bigg)^{k-1}.$$
Now,
$$ \langle \sigma_b y,  b^* \rangle = \langle y-2 \frac { \langle y, b \rangle}{|b|^2} b, b^* \rangle
=\langle y, b^* \rangle - 2 \frac{ \langle y, b \rangle}{|b|^2}=- \langle y, b^* \rangle $$
and we obtain
$$J_{b^*}\big (\sigma_b (y) \Big)=J_{b^*}(-y).$$
Therefore,
$$\widetilde \rho_{b^*}(x)= \frac{J_{b^*}(- \varphi_{b^*})(x)} {J_{b^*}(\varphi_{b^*})}=  \rho_{b^*}(x).$$
Furthermore,
$$
\begin{aligned}
&\rho_{b^*}(x)= \Bigg ( \frac {1-|b^*|^2}{|b^*-x|^2} \Bigg )^{k-1}=\Bigg ( \frac{|b|^2(|b|^2-1)}{|b-|b|^2 x |^2}\Bigg )^{k-1} \\
&=\Bigg( \frac{|b|^2-1}{|x-b|^2} \Bigg)^{k-1}=\rho_b(x),
\end{aligned}
$$
because $|b-|b|^2 x|^2=|b|^2- 2 |b|^2 \langle b, x \rangle +|b|^4=|b|^2 |x-b|^2,$ due to $|x|=1.$
Lemma is proved.

\end{proof}

\subsection{Ker $F_a$ for arbitrary center $a$}\label{S:arbitrary}
Now  Lemmas \ref{L:nub} and \ref{L:Rb} enable us to rewrite (\ref{E:kerFbnotfinal})
in the form similar to that for $F_a, |a|<1.$ Namely, we have
\begin{theorem}\label{T:kerFb}  Let $|b|>1.$ Then
$$\ker F_b= \{ f \in C(S^n): f(x)= - \rho_{b}(x) f(\tau_b x) \},$$
where the $b$-symmetry $\tau_b(x)$ is defined above  and
$$\rho_{b}(x)=\Big( \frac {|b|^2-1}{|x-b|^2} \Big)^{k-1}.$$
\end{theorem}

Thus, if we define $\rho_a,$ for all $|a| \neq 1,$ by
\begin{equation} \label{E:rab}
\rho_a(x)=\Big( \frac{  | 1-|a|^2| } {|x-a|^2} \Big)^{k-1},
\end{equation}
then  Theorems \ref{T:kerFa} and \ref{T:kerFb} can be combined in one statement, valid both for the cases interior and exterior centers :
\begin{theorem}\label{T:kerF} For any center $a \in \mathbb R^{n}, |a| \neq 1,$ holds
$$\ker F_a=\{ f \in C(S^n): f(x)= - \rho_a(x)f(\tau_a x) \},$$
where $\tau_a x$ is given by Lemma \ref{L:tau} and  $ \rho_a$ is defined by (\ref{E:rab}).
\end{theorem}
In the sequel, we will call function $f \in C(S^{n-1})$ $a-$ {\it odd} (or $a$-{\it even} )
if  $f(x) = -\rho_a(x)f(\tau_a x)$ (or $f(x)= \rho_a(x) f(\tau_a x)$), respectively) for all $ x \in S^{n-1}.$

\section{Paired Funk transforms and $T$-dynamics}

From now on, we start investigating the paired transform $(F_a, F_b)$ and its injectivity, using the obtained results on single transforms.
Our strategy is as follows. By Theorem \ref{T:kerF}, functions in the common kernel of $F_a$ and $F_b$ must satisfy two oddness conditions, with respect to both mappings $\tau_a, \ \tau_b.$ The composition of these two conditions implies certain invariance condition with respect
to the mapping $T=\tau_b \tau_a$ and all its iterations. Then understanding the behavior of the dynamical system, generated by the mapping $T: S^{n-1} \to S^{n-1}$  becomes a key tool for characterizing the common kernel $\ker F_a \cap \ker F_b.$

Thus, we start with defining the billiard-like self-mapping $T$ of $S^{n-1}$ as composition of the consequent symmetry mappings around $a$ and $b:$
\begin{equation}\label{E:Tab-def}
T=T_{a,b}:=\tau_b \tau_a : S^{n-1} \rightarrow S^{n-1}.
\end{equation}
The mapping $T_{a,b}$ is exactly the $V$-mapping $T$ defined by Definition \ref{D:T-def}. It acts as follows: one starts with a point $ x \in S^{n-1}$ and goes, till intersection with $S^{n-1}.$ along the straight line directed to $a.$
Then one proceeds along the straight line with the direction to $b$ and the next intersection with $S^{n-1}$ is the point $Tx.$
Since $\tau_a, \tau_b$ are involutions, the inverse mapping equals
$$T_{a,b}^{-1}=T_{b,a}=\tau_a \circ \tau_b.$$

\subsection{Fixed points of the mapping $T$}
The dynamics of the mapping $T$ is essentially characterized by its fixed points.
Denote
\begin{equation}\label{E:Z-def}
Z_{a,b}=\{ x \in S^n: \langle x, a \rangle = \langle x, b \rangle=1 \}.
\end{equation}
The set $Z_{a,b} \subset S^{n-1}$ is a subsphere of $\dim Z_{a,b} \leq n-3.$ It consists of all $x \in S^{n-1}$ such that
the points $a, b$ belong to the affine tangent plane $a, b \in T_x(S^{n-1}).$
\begin{lemma} \label{L:fixedpoints} Denote $Fix(T)$ the set of fixed points of the mapping $T.$ Then $Fix(T)=Z_{a,b} \cup (L_{a,b} \cap S^{n-1}).$
\end{lemma}
\begin{proof}
If $x^0 \in Z_{a,b}$ then the lines $L_{x^0, a}$ and $L_{x^0, b}$ are tangent to $S^{n-1}$ and meet $S^{n-1}$ at $x^0$ solely. This means that $\tau_a x^0=\tau_b x^0=x^0$ and  then $T x^0=\tau_b \tau_a x^0=x^0.$  In the case  $x^0 \in L_{a,b} \cap S^{n-1}=\{x^0, x^1 \},$ we have $\tau_a x^0=x^1, \ \tau_b x^1=x^0$ and then $Tx^0=\tau_b \tau_a x^0=x^0.$ Thus,  $Z_{a,b} \cup (L_{a,b} \cap S^{n-1}) \subset  Fix(T).$

Conversely, let $x^0 \in Fix(T),$  i.e., $x^0=\tau_b x^1, $ where $x^1=\tau_a x^0.$  Then  $x^0, x^1 \in L_{x^0, a} \cap L_{x^1, b},$ by Lemma \ref{L:tau}. If $x^0 = x^1$ then  $L_{x^0, a} \cap S^{n-1} =\{ x^0 \}, \ L_{ x^1, b} \cap S^{n-1} = \{ x^1 \}$ which means that the lines
$L_{x^0, b}, \ L_{x^0, b}$ are tangent to $S^{n-1}.$  Then $x^0 \in Z_{a,b}.$  Otherwise, if $x^0 \neq x^1$ then  $L_{x^0, a} = L_{x^1, b}$ since the two straight lines have two distinct common points. Therefore, the four  points $a, b , x^0, x ^1$ lie on the same line and hence $ x^0, \ x^1 \in L_{a,b} \cap S^{n-1}.$ Thus, in both cases $x^0 \in Z_{a,b} \cup (L_{a,b} \cap S^{n-1})$ and Lemma is proved.
\end{proof}
The following theorem shows that the mapping $T$ is deeply involved in the characterization of the common kernel of the transforms $F_a, F_b.$
\begin{theorem} \label{T:kerFab}
Let $f \in \ ker (F_a, F_b)= \ker F_a \cap \ker F_b.$  Then $f$ is $T$-automorphic, which means
\begin{equation}\label{E:T}
f(x)=\rho(x)f(Tx),
\end{equation}
where
$$\rho(x)=\rho_b(\tau_a x) \rho_a (x),$$
$\tau_a, \ \tau_b$ are described in Lemma \ref{L:tau} and  $\rho_a(x), \ \rho_b(x)$ are defined in (\ref{E:rab}).
\end{theorem}

\begin{proof} By Theorem \ref{T:kerF}, if $f \in \ker F_a \cap F_b$ then $f$ satisfies the two symmetry relations: $f(y)=-\rho_b(y) f (\tau_b y), \ f(x)=-\rho_a(x) f(\tau_a x).$
Substituting $y=\tau_a$ yields $f(x)=\rho_b (\tau_a x)f(\tau_b \tau_a x)=\rho(x) f(Tx).$
\end{proof}
The existence of nonzero  $T$-automorphic functions depends on the ergodic properties of the mapping $T=T_{a,b}.$
Our next goal is to understand how the  behavior of iterations of the mapping $T=T_{a,b}:S^{n-1} \to S^{n-1}$ and properties of the orbits
$$O_x =\{ T^k x, k=0,1,...\}$$
depend on the configuration of the centers $a, b.$

\subsection{Invariants $\Theta(a,b)$ and $\kappa(a,b)$}
For arbitrary two points  $a, b \in \mathbb R^{n}$ such that $|a|,|b| \neq 1,$ define
$$\Theta(a,b)=\frac{  \langle a , b \rangle -1}{\sqrt{(1-|a|^2)(1-|b|^2)} }.$$
This number can be real or purely imaginary. When $\Theta(a,b) \in [-1, 1]$ then the angle
$$\theta(a,b)=\arccos \Theta(a,b)$$ is defined. The number
$$\kappa(a,b)=\frac{\theta(a,b)}{\pi}$$
will be called {\it rotation number}. We will show that the dynamics of the mapping $T$ and the "size" of the common kernel $ker F_a \cap ker F_b$
can be fully characterized in terms of $\Theta(a,b)$ and $\kappa(a,b).$

\subsection{$T$-dynamics on 2-dimensional cross-sections}
\subsubsection{Complexification 2-dimensional cross-sections}
Fix $x_0 \in S^{n-1} \setminus Z_{a,b},$ where the singular set $Z_{a,b}$ is defined in (\ref{E:Z-def}). Consider any two-dimensional affine plane $\Sigma_{x_0}^2$ satisfying the condition
$$ x_0, a,  b \in \Sigma_{x_0}^2.$$
Of course, such a plane is unique unless $x_0, a, b$ belong to the same straight line.
By the definition of $Z_{a,b}$ the plane $\Sigma_{x_0}^2$ is not tangent to $S^{n-1}$ and hence its intersection with the unit sphere
$$C_{x_0}= \Sigma_{x_0}^2 \cap S^{n-1}$$
is a non-degenerate circle belonging to $S^{n-1}.$

The circle $C_{x_0}$  is invariant under the symmetries $\tau_a, \tau_b$ and
hence the $T$-orbit of $x_0$ entirely  belongs to the two-dimensional section:
$$O_{x_0}=\{T^k x_0 \}_{k=0}^{\infty} \subset C_{x_0}.$$

Let $c$ be the center of the open disc $\Sigma^2_{x_0} \cap B^n.$

The point $c$ belongs to $\Sigma_{x_0}^2 \cap B^n$  and hence
$$ \langle x-c, c \rangle =0, \ x \in \Sigma_{x_0}^2, $$
in particular, $\langle a-c, c \rangle = \langle b-c,  c \rangle =0.$

It will be convenient to identify the disc $\Sigma_{x_0}^2 \cap B^n$ with  the unit disc $\Delta$ in the complex plane.
For this purpose, introduce Cartesian coordinates in the disc $\Sigma^2_{x_0} \cap B^n$ by choosing an orthonormal basic $e_1, \ e_2$  in the linear space $\Sigma_{x_0}^2 -c.$ Then  any vector $x \in \Sigma_{x_0}^2$
can be written as $x=c+ \alpha_1 e_1+\alpha_2 e_2, \ \alpha _1, \alpha_2 \in \mathbb R. $

Now define the isomorphism $\zeta: \Sigma_{x_0}^2 \rightarrow \mathbb C$ by
\begin{equation} \label{E:dzeta}
z= \zeta_{x_0}(x)=\frac{ \alpha_1}{\sqrt{1-|c|^2}} + \frac{ \alpha_2 }{\sqrt{ 1-|c|^2}}i =\frac{ \langle x-c, e_1 \rangle}{\sqrt{1-|c|^2}} +  \frac{ \langle x-c, e_2 \rangle}{\sqrt{1-|c|^2}} i.
\end{equation}

Then $\zeta_{x_0}(\Sigma_{x_0} ^2 \cap B^n)=\Delta,$ where $\Delta$ is the unit complex disc, and $\zeta_{x_0}(\Sigma_{x_0} \cap S^{n-1})=\partial\Delta=S^1, \
\zeta _{x_0}(c)=0.$  Denote also
$$\zeta_{x_0}(a)=z_a, \ \zeta_{x_0}(b)=z_b.$$

We can transfer the mapping $T$ from the circle $C_{x_0}=\Sigma^2_{x_0} \cap S^{n-1}$ to $S^1$ by defining the new mapping:
$$T_{x_0} = \zeta_{x_0} T \zeta_{x_0}^{-1}:  S^1 \rightarrow S^1.$$

Denote $\langle z, w \rangle = Re \ \overline z w $ the real inner product in the complex plane.
Given two complex numbers $z, w$ define
$$\Theta(z, w)=\frac{ \langle z, w \rangle  -1}{ \sqrt{(1-|z|^2)(1-|w|^2) } }.$$

The following lemma justifies the term invariant applied to $\Theta(a,b).$
\begin{lemma} \label{L:Theta=Theta}
$\Theta(a,b)=\Theta(z_a, z_b).$
\end{lemma}

\begin{proof}
Let $ a=c+\alpha_1 e_1 + \alpha_2 e_2, \ b=c +\beta_a e_1 + \beta_2 e_2.$
Since  $\langle a-c, c \rangle=\langle b-c, c \rangle,$ then
$$
\begin{aligned}
&\langle a, b \rangle -1 = \langle a-c, b-c \rangle +|c|^2 -1 \\
&= (\alpha_1\beta_1+\alpha_2\beta_2)+|c|^2-1=(1-|c|^2)(\langle z_a, z_b \rangle -1).
\end{aligned}
$$
Similarly,
$$1-|a|^2=(1-|c|^2)(1- |z_a^2), \ 1-|b|^2=(1-|c|^2)(1-|z_b|^2).$$
Then the equality follows.
\end{proof}

\subsubsection {Induced  M\"obius transformations of the unit circle}

Thus, we have reduced the study of dynamics of the mapping $T$ on $S^{n-1}$ to the study of dynamics of the mappings $T_{x_0}$ on the circles $C_{x_0}=\Sigma_{x_0}^2 \cap S^{n-1}.$ It is supposed that $x_0 \notin Z_{a,b}$ so that the circle $C_{x_0}$ is non-degenerate and isomorphic to $S^1$ by means of the mappings $\zeta_{x_0}.$

Fix $x_0 \in S^{n-1} \setminus Z_{a,b}.$  If $ x \in C_{x_0}$ then $z=\zeta_{x_0} (x) \in S^1.$ It is clear that
$$\tau_{z_a}(z):=\zeta_{x_0}\tau_a \zeta_{x_0}^{-1}(z)$$ is the symmetry
of $S^1$ around the point $z_a=\zeta_{x_0}(a),$ having the same meaning as that in Lemma \ref{L:tau} with $n=2.$
 Therefore the mapping $T_{x_0}= \zeta_{x_0}  T \zeta_{x_0}^{-1}$ is the double reflection of the unit circle:
$$T_{x_0} (z)=  \tau_{z_b} \tau_{z_a} : S^1 \rightarrow S^1.$$

Of course, all the mappings $\tau_{z_a}, \tau_{z_b}$ depend on the point $x_0$ which is fixed in these considerations.

A simple calculation gives the explicit expression for $\tau_{z_a}:$
\begin{equation} \label{E:tauz}
\tau_{z_a} (z) =\frac{z_a-z}{1-\overline z_a z}, \ z \in S^1.
\end{equation}
Then $T_{x_0}(z)$ becomes a M\"obius transformation of the complex plane:
\begin{equation} \label{E:Tz}
T_{x_0}(z)=\tau_{z_b}(\tau_{z_a} (z) )=\frac{ (\overline z_a z_b -1)z + (z_a-z_b)}{ z(\overline z_a- \overline z_b)+ ( z_a \overline z_b -1)}.
\end{equation}

The group of M\"obius transformations of the complex plane is isomorphic to the group $PSL(2, \mathbb C)=SL(2, \mathbb C) / \pm I$ of unimodular complex matrices $M$ with elements $M$ and $-M$ identified.

The $M (T_{x_0}) \in PSL(2, \mathbb C)$ of the M\"obius transformation ( \ref{E:Tz}) is

\begin{equation}\label{E:MT}
M(T_{x_0})=\left[\begin{array}{cc}
\frac{\displaystyle{\overline  z_a z_b-1}}{\displaystyle{\sqrt{D}}} & \frac{\displaystyle{z_a-z_b}}{\displaystyle{\sqrt{D}}} \\
{} & {} \\
\frac{\displaystyle{\overline z_a -\overline z_b}}{\displaystyle{\sqrt{D}}}& \frac{\displaystyle{z_a \overline z_b -1}}{\displaystyle{\sqrt{D}}}
\end{array} \right],
\end{equation}
where
$$D = \det \ M(T_{x_0})= (1- |z_a|^2)(1-|z_b|^2). $$

Also, straightforward computation based on (\ref{E:dzeta}) shows
 that if $z =\zeta_{x_0} (x),  \ x \in C_{x_0}$ then
\begin{equation}\label{E:rhoaz}
\rho_a(x)=\Bigg(  \frac{|1-|a|^2|}{|x-a|^2} \Bigg)^{k-1}=\Bigg(\frac{|1-|z_a|^2| }{ |z-z_a|^2 } \Bigg) ^{k-1}= |(\tau_{z_a})^{\prime} (z)|)^{k-1}
\end{equation}
where $\tau_{z_a}^{\prime}(z)$ is the complex derivative of the function $\tau_{z_a}(z)$ in (\ref{E:tauz}).
Then $\rho(x)=\rho_b (\rho_a(x))$ takes in this model the form
\begin{equation} \label{E:rhoz}
\rho(x)=\rho_b(\tau_a x) \rho_a(x) = | \big ( |\tau_b \circ \tau_a)^{\prime}(z)| \Big )^{k-1}=|T_{x_0}^{\prime}(z)|^{k-1},
\end{equation}
where $x \in \Sigma^2_{x_0} \cap S^{n-1}$ and $z \in S^1$ are related by $z = \zeta_{x_0}x.$  We will write also  $\rho_{x_0}(z) = \rho \big(\zeta^{-1}_{x_0}(z) \big),$ so that
the left hand side in (\ref{E:rhoz}) is $\rho(x) = \rho_{x_0}(z).$

Recall that here $T_{x_0}(z)$ is the complex M\"obius transformation, preserving the unit circle $S^1$ and such that $T_{x_0}(\zeta_{x_0}(x))=Tx, \
x \in  C_{x_0}=\Sigma_{x_0}^2 \cap S^{n-1}.$

\subsection{Classification of the $T_{x_0}$-orbits on the unit circle} \label{S:classification}

M\"obius mappings $T(z)$ in the complex plane and the behavior of their orbits  are classified in terms of  the trace $tr M (T)$ of representing unimodular matrices
$M(T) \in PSL(2,\mathbb C)$  and the rotation number
$$\kappa_{M (T)}=\frac{1}{\pi}\arccos \frac {1}{2} tr \ M (T).$$

The standard information about  M\"obius mappings in the complex plane, which we will need,  can be found, for example, in \cite[Chapter 1, Sections 8-10]{Ford}. An important role in the classification play fixed points $\z^0, \ z^1$ of the transform $T.$  The eigenvalues $\lambda_0=K, \ \lambda_1= \frac{1}{K}$ ($K$ is called multiplier) are the values of the derivative at the fixed points: $\lambda_0=T^{\prime}(z^0), \ \lambda_1=T^{\prime}(z^1)$ .

In our case, we consider M\"obius mappings, preserving the boundary $S^1=\partial \Delta$ of the unit complex disc. According to the classification (see, e.g.,\cite [Theorem 15]{Ford}, the mapping $T:S^1 \rightarrow S^1 $ belongs to  one of the following types:
\begin{enumerate}
\item Hyperbolic: $ tr \ M(T)  \in \mathbb R$  and $|tr M_T| > 2.$ There are two, attracting and repelling, fixed points on $S^1.$
\item Parabolic: $tr \ M(T)=\pm 2. $ There is one attracting fixed point on $S^1.$
\item Elliptic: $tr \ M(T)  \in \mathbb R$  and $|tr \ M(T)| <2.$ There is no fixed points on $S^1$ , the orbits are dense if the rotation number $\kappa_{M(T)} $ is
irrational, otherwise $M (T) $ has finite order and all orbits are finite.
\item Loxodromic: $Im \ tr \ M(T) \neq 0.$  There are two, attracting and repelling,  fixed points on $S^1.$
\end{enumerate}

To understand the dynamics of the mapping $T_{x_0},$ we need to compute the trace of the matrix $M ( T_{x_0} ) .$ We have from (\ref{E:MT}):
\begin{equation}
\begin{aligned}
&tr \ M (T_{x_0})=\frac{ 2 Re (z_a\overline z_b)-2}{ \sqrt{|z_a \overline z_b -1|^2-|z_a- z_b|^2} } \\
&= 2 \frac{ \langle z_a, z_b \rangle -1}{ \sqrt{(1-|z_a|^2)(1-|z_b|^2)} } =2 \Theta(z_a,z_b).
\end{aligned}
\end{equation}

Thus, by Lemma \ref{L:Theta=Theta} the trace of $M ( T_{x_0} )$ is
\begin{equation} \label{E:trace}
tr \ M (T_{x_0}) =2\Theta(a,b)
\end{equation}
and the rotation number is
\begin{equation} \label{E:rotation}
\kappa_{T_{x_0}}(z_a , z_b):= \frac{1}{\pi} \arccos \Theta(z_a,z_b)= \frac{1}{\pi}\arccos \Theta(a,b)=\kappa(a,b).
\end{equation}

\begin{remark}\label{R:the_same}
\begin{enumerate}[(i)]
\item
Formulas (\ref{E:trace}), (\ref{E:rotation}) show that the numbers \\
$tr \ M (T_{x_0}) $ and $\kappa_{T_{x_0}}(z_a, z_b)$  are independent of the choice of the point
 $x_0 \in S^{n-1} \setminus Z_{a,b}.$ That means that the type of the dynamics of the restrictions of the mapping $T$ on the circles $C_{x_0}=\Sigma^2_{x_0} \cap S^{n-1}$ is determined by the invariants $\Theta(a, b)$ and $\kappa(a,b) $ and this type the same for all $x_0 \in S^{n-1},$  except for the singular set $Z_{a,b}$ of codimension $\geq 2.$
\item
 In the elliptic case with $\kappa(a,b)=\frac{p}{q}, \ p$ and $q$ are coprime, all the mappings $T_{x_0}$ are periodic with the same period $q$ and hence the mapping $T$ has the same property, i.e., $T^q=id.$
\end{enumerate}
\end{remark}

\subsection {$T$-automorphic functions on circles}

Let $ f \in C(S^{n-1})$ be a $T$-automorphic function on the unit sphere, which means
$$f(x)=\rho(x) f(Tx),  \ x \in S^{n-1}.$$
Fix again $x_0 \in S^{n-1} \setminus Z_{a,b}.$ Then $C_{x_0}=\Sigma_{x_0}^2 \cap S^{n-1}$ is a non-degenerate circle. As it was already mentioned, this circle is invariant under the mapping $T,$ i.e., $T(C_{x_0})= C_{x_0}.$

The mapping $\zeta_{x_0}: C_{x_0} \rightarrow S^1$ allows to transfer the function  $f(x)$ defined on $C_{x_0} \subset S^{n-1}$   to the circle $z \in S^1 \subset \mathbb C,$  by introducing the new function
$$ f_{x_0}(z)=f(\zeta^{-1}_{x_0}(z) ), \  z \in S^1.$$

We have shown that the dynamics of $T\vert_{C_{x_0}}$ reduces to study of dynamics of the complex M\"obius mapping
$$ S^1 \ni z \rightarrow T_{x_0}(z) \in S^1.$$
By the construction of $T_{x_0}(z), z \in S^1,$ and formula (\ref{E:rhoz}),
the $T$-automorphic function $f(x)$ on $S^{n-1}$ transforms to a $T_{x_0}$ -automorphic function $f_{x_0}(z)$ on $S^1,$ i.e. the relation holds:
\begin{equation} \label{E:f(z)}
f_{x_0}(z)=\rho_{x_0}(z)f_{x_0}(T_{x_0}z)=|T_{x_0}^{\prime}(z)|^{k-1} f_{x_0}(T_{x_0}z).
\end{equation}
Now by iterating  (\ref{E:f(z)}) $N$ times, we have from the chain rule:
\begin{equation}\label{E:iterate}
\begin{aligned}
&f_{x_0}(z)=|T_{x_0}^{\prime}(z)|^{k-1} f_{x_0} \big (T_{x_0}(z) \big) \\
&=\prod_{j=0}^{N-1} |T_{x_0}^{\prime} \big(T_{x_0}^{\circ j}(z) \big )|^{k-1}
f \big(T_{x_0}^{\circ N} (z) \big )
= |(T_{x_0}^{\circ N})^{\prime}(z)|^{k-1} f_{x_0} \big( T_{x_0}^{\circ N} (z) \big).
\end{aligned}
\end{equation}

\begin{proposition} \label{P:dynamics} Suppose that $f_{x_0}(z)$ is a $T_{x_0}-$ automorphic, i.e., satisfying (\ref{E:f(z)}), continuous function on $S^1,$   .
\begin{enumerate}[(\rm i)]
\item If the mapping $T_{x_0}:S^1 \rightarrow S^1$ is hyperbolic, parabolic, or loxodromic then $f_{x_0}=const$ for $k=1.$ Moreover, $f_{x_0}=0$ for $k>1.$
\item If $T_{x_0}$ is elliptic with  the irrational rotation number $\kappa_{T_{x_0}}(a,b)$ and $f_{x_0}$ has zeros on $S^1,$ then $ f_{x_0}=0.$
\item If $T_{x_0}$ is elliptic and the rotation number $\kappa_{T_{x_0}}$ is rational then there are nonconstant automorphic functions $f_{x_0}$.
\end{enumerate}
\end{proposition}

\begin{proof} Statement (i ) splits into the three cases: hyperbolic, parabolic and loxodromic, which will be considered separately.

\medskip

\noindent
{\it Hyperbolic case.}

If $T_{x_0}$ is hyperbolic then $T_{x_0}$ has  two fixed points $z^0,z^1 \in S^1,$ one of them, say, $z^0$ is attracting and another one is repelling.
Pick $z \in S^1, z \neq z^1.$
Let $z_j=T_{x_0}^{\circ j} (z).$  The point $z^0$ is attracting, hence
$$z_j \to z^0.$$
Then $$ \ T_{x_0}^{\prime}(z_j) \to T_{x_0}^{\prime}(z^0), \ j \to \infty.$$
The eigenvalue $\lambda_0$  corresponding to the attracting fixed point satisfies
$ 0< \lambda_0 = T_{x_0}^{\prime}(z^0)<1$ \cite[Section 7]{Ford}.
Therefore, if $k>1$ then the infinite product converges to zero:
$$\prod\limits_{j}^{\infty}|T_{x_0}^{\prime} \big (T_{x_0}^{\circ j} (z) \big)|^{k-1}=0.$$

Also $f_{x_0} \big (T_{x_0}^{\circ N} (z) \big )  \to f(z^0), \ N \to \infty.$
Letting $N \to \infty$ in (\ref{E:iterate}) yields $f_{x_0}(z)=0, z \neq z^1.$ By continuity, $f_{x_0}(z)=0$ everywhere on $S^1$ for $k>1.$ If $k=1$ then we have
$f_{x_0}(z)=f \big(T_{x_0}^{\circ N} (z) \big)$ and letting $N \to \infty$ implies $f_{x_0}(z)=f_{x_0}(z^0)=const.$

\medskip

\noindent
{\it Parabolic case}.

In this case, there exists only one, attracting, fixed point $z^0 \in S^1$  with $\lambda_0=\lambda_a = T^{\prime} (z^0) =1.$
Any parabolic mapping $T_{x_0}: \Delta \to \Delta$ is conjugated to the mapping
$$ S_{t}: w \rightarrow w+t, \ t  \in \mathbb R, t \neq 0,$$
of the upper halfplane $ H^+:=\{Im \ w >0\}.$ This means that there exists a conformal mapping
$$\Psi: \Delta \rightarrow H^+,  \ \Psi(z^0)=\infty, $$
such that
$$ T_{x_0}=\Psi^{-1} \circ S_t \circ \Psi.$$

Then
$$T_{x_0}^{\circ N}=\Psi^{-1} \circ S_{Nt} \circ \Psi $$ and
$T_{x_0}^{\circ N}(z)=\Psi^{-1}( \Psi  (z) + Nt).$
Then
$$(T_{x_0}^{\circ N} )^{\prime}(z)=\frac{ \Psi^{\prime}(z) } { \Psi^{\prime} \Big( \Psi^{-1} \big(\Psi(z) + Nt \big) \Big) }.$$

Pick $ z \neq z^0.$ Letting $N \to \infty$ yields
$$
( T_{x_0}^{\circ N} )^{\prime} (z) \to  \frac{\Psi^{\prime}(z)}{ \Psi^{\prime}(\Psi^{-1}(\infty))}=\frac{\Psi^{\prime}(z)}{\Psi^{\prime}(z^0)}.
$$
However, $\Psi^{\prime}(z)$ is finite, while $\Psi$ has a pole at $z_0$ and hence $\Psi^{\prime}(z^0)=\infty.$
Thus,
$$\lim\limits_{N \to \infty} (T_{x_0}^{\circ N} )^{\prime}(z) =0, \ z \in S^1 \setminus \{ z^0\}.$$

If the dimension $k$ of planes in Funk transform is $k>1$ then letting $N \to \infty$ in (\ref{E:iterate}) yields $f_{x_0}(z)=$ for all $z \in S^1$ but one, and hence $f_{x_0}=0$ identically. If $k=1$ then we obtain $f_{x_0}(z)=f_{x_0}(z^0)=const.$

\medskip

\noindent
{\it Loxodromic case.}

The loxodromic case is similar to the hyperbolic one. It corresponds to non-real multipliers $K=\lambda_0.$ There are two fixed points, attracting and repelling, and the eigenvalue $\lambda_0 =  T_{x_0}^{\prime}(x^0)$ at the attracting point satisfies $|\lambda_0|< 1.$ Then, like in the hyperbolic case,  (\ref{E:iterate}) implies $f_{x_0}=0$ when $k>1$ and $f_{x_0}=const$ when $k=1.$

\medskip

\noindent
({\it ii} ) {\it Elliptic case. Irrational $\kappa (a,b).$}

If $T_{x_0}$ is of elliptic type then $T_{x_0}$ is conjugate with a rotation $U_{\psi}(z)=e^{i \theta}z:$
$$T_{x_0} =g \circ U_{\theta} \circ g^{-1}.$$
The angle of rotation $\theta$ is given by $ \theta= 2 \arccos \frac{1}{2} tr \ M(T_{x_0})$
[ cf. \cite[Section 8]{Ford}, and  the rotation number is $\frac{\theta}{2\pi}.$
Since by (\ref {E:trace} ) $ \Theta(a,b)= \frac{1}{2} tr \ M(T_{x_0}),$ we have $\frac{\theta}{2 \pi}=  \kappa (a,b).$

If the rotation number $\kappa (a,b)$ is irrational then the orbit $\{T_{x_0}^{\circ j} z \}_{j=0}^{\infty}$ of any point $z \in S^1$ form an irrational wrapping of $S^1$ and is  dense in $S^1.$ If $f_{x_0}(e)=0$ for some $e \in S^1$ then (\ref{E:f(z)}) implies, by iterating, that $f_{x_0}(T_{x_0}^{N}e)=0$ for all $N=0,1,...$ and since the orbit of $e$ is dense, then by continuity $f_{x_0}(z)=0$ for all $z \in S^1.$

\medskip

\noindent
( {\it iii } ) { \it Elliptic case. Rational $\kappa (a,b).$}

In this case $T_{x_0}$ is periodic. Namely,
if $\kappa(a,b) = \frac{p}{q}, $  then the mapping $T_{x_0}$ is of order $q,$ i.e., the $q$-th iteration $T_{x_0}^{\circ q}=id.$
\begin{lemma}\label{L:W^q}  Define the operator  $Wh (z)= \rho_{x_0}(z) h(T_{x_0}z), \ h \in C(S^1),$
where $\rho_{x_0}(z)=|T_{x_0}^{\prime}(z)|^{k-1}.$
If $T_{x_0}^{\circ q}=id$ then $W^q=I.$
\end{lemma}
\begin{proof} By the chain rule, $(W^{q}h)(z)=|(T_{x_0}^{\circ q} )^{\prime}(z)|^{k-1} h \big (T_{x_0}^{\circ q }(z) \big ).$
Since $T_{x_0}^{\circ q} (z)=z$ then $(T_{x_0}^{\circ q} )^{\prime}(z)=1$ and $(W^q h)(z)=h(z).$
\end{proof}

\begin{remark} In fact, Lemma \ref{L:W^q} reflects the fact that the mapping $T \to W,$ which takes a mapping $T$ to the operator $W$ is a representation of the subgroup of $PSL(2,\mathbb C),$ preserving $S^1,$ in the space of the operators on $C(S^1).$
\end{remark}
Now we are able to construct a nonconstant $T_{x_0}-$ automorphic function $g$ as follows. Let $h \in C(S^1)$ be arbitrary and
$$g(z)=\sum_{j=0}^{q-1} W^j h(z),$$
where we have denoted
$$(Wh)(z)=\rho_{x_0}(z)f_{x_0}(T_{x_0}z)= |T_{x_0}^{\prime}(z)|^{k-1}f_{x_0}(T_{x_0}z).$$

Then $W^q=id$ implies $Wg=g$ and hence $g$ is $T_{x_0}-$ automorphic. The function $g$ can be chosen nonconstant. Indeed, pick a point $e \in S^1$ and
the function $h$ such that $h(z) \neq \const$  near $e$ and $supp h \subset U_e,$ where $U_e$ is  a neighborhood of $e$ such that $T_{x_0}e, T_{x_0}^2 e,..., T_{x_0}^{j-1}e \notin U.$ Then $g(z)=h(z)$ in a small neighborhood of $e$ and since $h \neq \const$ there then $g \neq \const.$
The proof is complete.
\end{proof}
\subsection{$T$-dynamics on $S^{n-1}$ } \label{S:geometric}

According to Remark \ref{R:the_same}, the types of dynamics on the sections $\Sigma_{x_0}^2 \cap S^{n-1}$ are the same for all $x_0 \in S^{n-1}$ except for a subsphere $Z_{a,b} \subset S^{n-1}$ of codimension at least two (see (\ref{E:Z-def}).

The classification  of the types of the mapping $T: S^{n-1} \rightarrow S^{n-1} $ has a clear geometric meaning. By Lemma \ref{L:fixedpoints}, the set of fixed points of $T$ is $Fix(T)= Z_{a,b} \cap (L_{a,b} \cap S^{n-1}).$   In the {\it hyperbolic}  and {\it loxodromic} cases the line $L_{a,b}$ meets $S^{n-1}$ at two $T$-fixed points. The difference between the two cases is that in the hyperbolic case the points $a$ and $b$ are on one side from the unit sphere,  i.e., $(1-|a|^2)(1-|b|^2) >0,$  while in the loxodromic case they are separated by $S^{n-1},$ i.e., $ (1-|a|^2)(1-|b|^2) <0.$
The {\it parabolic} case (one fixed point) corresponds to the limit case  when $L_{a,b}$ is tangent to $S^{n-1}$ and the two fixed points merge. If $L_{a,b}$ is disjoint from $S^{n-1}$ then we deal with the {\it elliptic} case.

\section{Proofs of main results}
\subsection{Proof of Theorem \ref{T:main}}
Define the operators
$$W_a f(x)= \rho_a(x) f(\tau_a x), \ W_b f(x) =\rho_b(x) f(\tau_b x), \ \ f \in C(S^{n-1}) $$
It follows from (\ref{E:rho-tau}) that $W_a$ and $W_b$ are  involutions: $W_a^2 f= W_b^2 f = f.$

By Theorem \ref{T:kerF} a function $f \in \ker F_a \cap \ker F_b$ of and only if $W_a f = -f,  \ W_b f= - f.$
Define
\begin{equation}\label{E:W-def}
(Wf)(x)=(W_a W_bf)(x)= \rho(x)f(Tx),
\end{equation}
where
$$\rho(x)=\rho_b(\tau_a x)\rho_a(x), \ Tx=\tau_b(\tau_a x)$$ then $f=Wf$ for $f \in C(S^{n-1}),$  (see (\ref{E:T}) ).
By Theorem \ref{T:Fab} any  $ f \in \ker F_a \cap \ker F_b$ is a $T$-automorphic function, i.e,:
$$ f(x)=\rho(x)f(Tx), \ x \in S^{n-1}.$$
Observe that $\rho_a(x) >0 $ and $f(x)=-\rho_a(x)f(\tau_a x),$ imply that the function $f$ changes sign at the symmetric points $x, \ \tau_a x$ and since $f$ is continuous, it has zeros on any circle $C_{x_0}=\Sigma_{x_0}^2 \cap S^{n-1},$ for all $x_0 \in S^{n-1}$ except a sub-sphere $Z_{a,b} \subset S^{n-1}$ of codimension two.

Proposition \ref{P:dynamics} and Remark \ref{R:the_same}  imply the dichotomy: either $f_{x_0}=0$ for any $x_0 \in S^{n-1} \setminus Z_{a,b}$ and then $f=0$,  or all the mappings $T_{x_0}$ have the elliptic type with  $ \kappa(a,b) \in Q.$  In this case $T_{x_0} $ are periodic. By (\ref{E:rotation}) (see Remark \ref{R:the same} (ii)} the order of periodicity is the same for all the mappings $T_{x_0}$ and is defined by the rotation number $\kappa(a,b)=\frac{p}{q}.$ Then the mapping $T$ is periodic of order $q.$

It remains to prove in the latter case $ker F_a \cap ker F_b \neq 0.$  Since all functions in $ker F_a \cap ker F_b$ are $T$-automorphic, we first construct, using the periodicity of $T$, a nonzero $T$-automorphic function, similarly to what we did in Proposition \ref{P:dynamics}.

This function is not guaranteed to belong to $ker F_a \cap ker F_b$ , but its $a$-odd part does. Then we modify the function in such a way that the above $a$-odd part is not identically zero.

So, let $T$ be periodic, $T^q=id.$ Then Lemma \ref{L:W^q} and formula (\ref{E:rhoz}) imply $W^q=I.$  We assume $q$  the minimal possible number here.
Choose an arbitrary function $h \in C(S^n)$ and define
$$g=\sum_{j=0}^{q-1}W^j h.$$
Since $W^q=I,$ we have  $g=Wg,$ i.e.,  $g$ is automorphic. Now apply $W_a$ to both sides of this identity
$$W_a g= W_a W g=W_a^2 W_b g= W_b g$$
and define
$$f=g-W_a g=g-W_b g.$$
Then $f$ satisfies both relations $W_a f=-f$ and $W_b f =-f,$ i.e. $ f \in \ker F_a \cap \ker F_b.$

\begin{lemma}\label{L:e}
There exists $ e \in S^{n-1}$ such that $e \notin \tau_a( O_e).$
\end{lemma}
\begin{proof}
If the assertion of Lemma fails to be true then for any $e \in S^{n-1}$ there exists $ 0 \leq j \leq q-1$ such that $e=\tau_a (T^{\circ j} e).$
Since $T$ is real-analytic, standard argument shows that $j$ can be taken independent of $e,$ i.e.,  $\tau_a = T^{\circ j}$ on $S^{n-1}.$ Then $T^{2j}=\tau_a\tau_a = id.$

On the other hand, $\tau_b T^{\circ j}=\tau_b \tau_a =T.$  This implies $T^{j-1}=\tau_b$  and hence $T^{2j-2}=\tau_b \tau_b=id.$
Then $T^2=T^2 T^{2j-2}=T^{2j}= id.$
Thus, $q=2, j=1$ and we obtain $\tau_a =T$ or, the same, $\tau_a=\tau_b \tau_a.$ Then $\tau_b=id$ which is not the case. This contradiction completes the proof.
\end{proof}

The next step is to prove that the function $f$ can be chosen to be nonzero. By Lemma \ref{L:e} there exists $e \in S^{n-1}$ such that $e \notin \tau_a(O_e).$
The orbit $O_e=\{ e, Te, ..., T^{q-1}e \}$ is finite. Therefore, we can choose a small neighborhood $V_e$ of $e$ so that
$$V_e \cap (O_e \setminus \{e\}) = \emptyset, \  V_e \cap \tau_a (O_e) = \emptyset.$$

Now, if, from the beginning, we provide $\supp h \subset V_e$ with $h(e) \neq 0$ then $(W^j h)(x)=0$ for $j=1, ..., q-1$ and $ x \in V_e.$
Also $g(e)=h(e)\neq 0.$
By the construction, $g(\tau_a e)=0,$ because $\tau_a e \in \tau_a(O_e)$ and $\tau_a(O_e) \cap \supp g = \emptyset.$  Then
$f(e)=g(e)-\rho(e) g(\tau_a e)=h(e) \neq 0.$  Thus, we have constructed, for the case of periodic $T$ and $W,$ a nonzero function $f \in \ker F_a \cap \ker F_b,$ which completes the proof of Theorem \ref{T:main}.

\subsection{Proof of Corollary \ref{C:finite}}
Denote $\tau_1=\tau_a, \ \tau_2= \tau_b.$ We regard $\tau_i, i=1,2$ as elements of the group $G.$ Then $T=\tau_2\tau_1$ - the product in $G.$
Then $\tau_i^2=e, \ i=1,2 ,$ where $e $ is the unit element of $G.$ If $G$ is finite then  $T$ is an element of finite order, i.e. the mapping $T$ is periodic and the condition of Theorem \ref{T:main0} is fulfilled.  Conversely, if $T$ is periodic, then $T^q=e$ for some $q \in \mathbb N.$ In this case the length of any irreducible word in $G$ does not exceed $q$ and hence $G$ is finite. Thus, the conditions in Theorem \ref{T:main0} and Corollary \ref{C:finite} are equivalent.

\subsection{Proof of Theorem \ref{T:main0}}
As we saw in the proof of Theorem \ref{T:main}, $ker F_a \cap ker F_b \neq 0$ if and only if the mapping $T$ is of elliptic type with rational rotation number $\kappa(a,b).$ Since this condition is equivalent to periodicity of $T,$ Theorem \ref{T:main0} immediately follows.

\subsection{Proof of Theorem \ref{T:main-geom}}
Theorem \ref{T:main-geom} is  a reformulation of Theorem \ref{T:main} in geometric terms. Consider the straight line $L_{a,b}$ through $a$ and $b.$
Then  $L_{a,b} \cap S^{n-1} \neq \emptyset$ if and only if the equation $|a+t(b-a)|^2=1$
has a real solution $t$ which is equivalent to the condition for the discriminant of the corresponding quadratic equation for $t$:
$$
\begin{aligned}
&\langle a, b-a \rangle -|a-b|^2(|a|^2-1)=( \langle a, b \rangle -1 )^2 - (1-|a|^2)(1-|b|^2) \\
&=(1-|a|^2)(1-|b|^2)(\Theta^2 (a,b) -1) \geq 0.
\end{aligned}
$$
It is satisfied when either $(1-|a|^2)(1-|b|^2) > 0$ and $|\Theta(a,b)| \geq 1$, which corresponds to the hyperbolic or parabolic  case, or $(1-|a|^2)(1-|b|^2) <0$ and
$\Theta^2 (a,b ) \leq 1.$  In the latter case, $\Theta(a,b)=\frac{ \langle a, b \rangle -1} { \sqrt{(1-|a|^2)(1-|b|^2)}}, $ is purely imaginary, $\Theta^2 (a,b)<0,$ and hence $T$ is loxodromic, unless
$\langle a, b \rangle -`1=0$ when $T$ is periodic of order $2.$

Thus, if $ \langle a, b \rangle = 1$ then $\ker F_a \cap \ker F_b \neq \{0\}.$ Otherwise, the injectivity holds for $T$ of hyperbolic, parabolic or loxodromic types, corresponding to  $L_{a,b} \cap S^{n-1} \neq \emptyset,$ and  for $T$ elliptic type,  corresponding to $L_{a,b} \cap S^{n-1} =\emptyset,$
with irrational rotation number $\kappa (a,b).$ These are exactly all the injectivity cases enlisted in Theorem \ref{T:main-geom}.
Proof is complete.

The configurations of the centers $a, b$ and the types of $T-$ dynamics corresponding to injective pairs $F_a, \ F_b$   are shown on Fig.2.
\begin{figure}
\centering
\includegraphics[width=12cm]{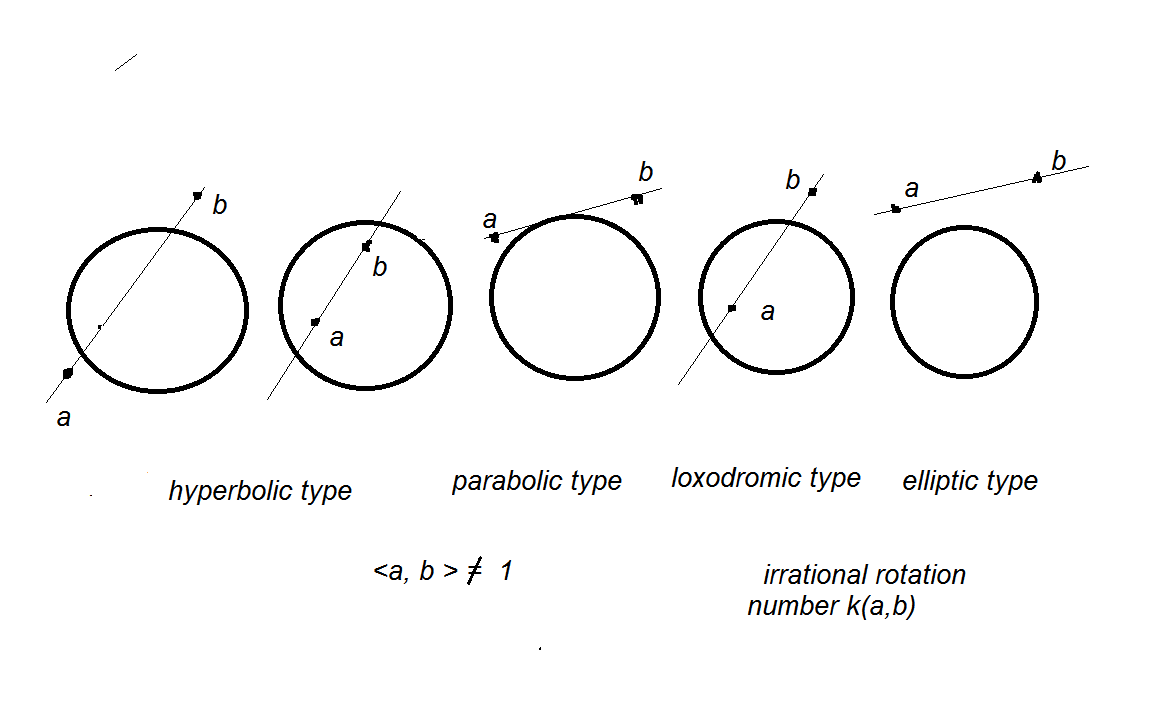}
\caption{The $T$-dynamics types corresponding to the injectivity of the paired transform $(F_a, F_b)$ }
\end{figure}

\section{Generalizations and open questions}

\subsection{Paired Funk transforms with centers at $\infty$ }
The transform $\Pi_b$ (\ref{E:Pib}) ,
corresponding to the center at infinity, can be also included in our considerations. Arguments, similar to those we have
used for the transforms $(F_a,F_b)$, lead to
\begin{theorem}\label{T:add}
\begin{enumerate}[(\rm i)]
\item The paired transform $(F_a, \Pi_b), |a| \neq 1, \ b \neq 0,$ fails to be injective if and only if $\langle a, b \rangle^2 \leq |b|^2(|a|^2-1)$ and
$\frac{1}{\pi} arccos \frac{\langle a, b \rangle }{|b| \sqrt{|a|^2-1}}$ is rational.
 \item The paired transform $(\Pi_ {b_1}, \Pi_{b_2}), \ b_1, b_2 \neq 0,$ fails to be injective if and only if the angle $\angle( b_1 , b_2)=\arccos\frac{\langle b_1 , b_2 \rangle}{|a||b|}$ between the vectors $b_1$ and $b_2$ is a rational multiple of $\pi.$
\end{enumerate}
\end{theorem}
Formally, Theorem \ref{T:add} can be obtained from Theorem \ref{T:main} by replacing $b$ by $\lambda b$ (in the case (i)), or, in the case (ii), $b_1$ and $b_2$ by $\lambda b_1, \lambda b_2,$ respectively, and  letting $\lambda \to  \infty.$

\subsection{Multiple Funk transforms}
Consider a $s$-tuple of points $A=\{ a_1, ..., a_s\}, \ |a_j| \neq 1, $ and ask the similar questions: for what sets $A$ of centers the condition
$$ ker F_{A}:= \cap_{j=1}^s ker F_{a_j}=\{0\}$$
holds?  Here the transform $F_A$ is understood as $F_A f =(F_{a_1}f, ..., F_{a_s}f), \ f \in C(S^{n-1}).$

A sufficient condition immediately follows from Theorem \ref{T:main0}:
\begin{theorem} \label{T:FA} Suppose that there are two centers $a_i, a_j$ such that the $V$-mapping $T_{i,j}=\tau_{a_i} \tau_{a_j},$ where the symmetries $\tau_a$ are defined in Lemma \ref{L:tau}, is non-periodic. Then  $ker F_A=\{0\}.$
\end{theorem}
In particular, the equivalent Theorem \ref{T:main-geom} implies injectivity of the multiple transform $F_A$ if at least one center $a_j$ lies inside the unit sphere $S^{n-1}.$

Denote $G(A)$ the group generated by the symmetries $\tau_i: = \tau_{a_i}, \ i=1,..., s.$  Theorem \ref{T:FA} says that if $\ker F_A \neq \{0\}$ then
all $T_{i,j} \in G$ are elements of finite order, i.e., $(\tau_i \tau_j)^{q_{i,j}} = T_{i,j}^{q_{i,j}} = e,$ where $e=id$ is the unit element in $G(A).$  Also, $q_{i,i}=1$ because $\tau_i$ are involutions. The groups $G(A)$ with the above identities for generators are called (abstract) {\it Coxeter groups} ( cf., \cite{Bjorn}, 1.1 ). Thus, we have
\begin{corollary} If $ker F_{A} \neq \{0\}$ then $G(A)$ is a Coxeter group.
\end{corollary}

{\bf Question}
\begin{enumerate}[(\rm i)]
\item
{\it Describe all  set $A =  \{a_j, \ 1 \leq j \leq s\}$ such that $\ker F_A = \{0\}.$
\item Can necessary and sufficient conditions of injectivity of the multiple transform $F_A$ be formulated
in terms of the group $G(A)?$
\item In  particular, is the converse statement to Theorem \ref{T:FA} true, i.e., is it true that
$ker F_A \neq \{0\}$ whenever $G(A)$ is a Coxeter group?}
\end{enumerate}
For $s=2,$ the answers are given in equivalent Theorems \ref{T:main0}, \ref{T:main}, \ref{T:main-geom} and Corollary \ref{C:finite}.   Similar questions are applicable to the case of infinite centers, when the transforms $F_{a_j}$ or part of them are replaced  by parallel slice transforms $\Pi_{b_j}.$ In this case, by Lemma \ref{L:kernels}, the associated symmetry is $\sigma_b$ - the reflections across the hyperplanes $\langle x, b \rangle =0.$ The following theorem generalizes Theorem \ref{T:add}(ii) and gives a complete answer to the above questions for parallel slice transforms.
\begin{theorem}\label{T:Coxeter} Let $B= \{ b_1,..., b_s \}$ be a finite system of distinct unit vectors in $\mathbb R^n.$ The multiple transform $\Pi_B= (\Pi_{b_1}, ..., \Pi_{b_s})$ fails to be injective, i.e.,  $ker \Pi_B := \cap _{j=1}^s \ker \Pi_{b_j} \neq \{0\} $ if and only if the group $G(B)$ generated by the reflections $\sigma_{b_j}, \ j =1, ...,s $ is a finite Coxeter group.
 \end{theorem}

 \begin{proof} First of all, notice that finite groups generated by reflections are Coxeter groups.
 Suppose that the group $G(B)$ is finite. Denote $\P_j=\{ \langle x, b _j \rangle =0\}. $
 Let $\P_1, ..., \P_N$ be the complete system of mirrors of the reflection group $G(B)$, i.e. the system of hyperplanes obtained from $\P_1, ...,\P_s$
 by applying arbitrary elements $g \in G(B).$  Let $b_1,..., b_N$ be a corresponding system of normal vectors, which contains the system $B.$  The reflections $\sigma_{b_i}, i=1,..., N,$ map the system $\{\P_j \}_{j=1}^N$ of the hyperplanes onto itself
 and transform the system $\{b_j\}_{j=1}^{N}$ of normal vectors into the system $\{ \pm b_j \}_{j=1}^{N},$ with an odd  number of the signs minus.

Therefore, if we define $$f(x)=\prod\limits_{j=0}^N \langle x, b_j \rangle,$$
Then $f (\sigma_{b_i} x)= \prod\limits_{j=0}^N \langle \sigma_{b_i} x,  b_j \rangle=\prod\limits_{j=0}^N \langle x, \sigma_{b_i}  b_j \rangle = - f(x),$  Therefore $f \in ker \Pi_{b_i}$ for all $i=1. ..., N$ and hence $f \in \cap_{i=1}^N ker \Pi_{b_i}.$ Obviously, $ f \neq 0.$

Conversely, suppose that the group $G(B)$ is infinite. Then it possesses an infinite system of mirrors $\{\P_i\}_{1}^{\infty}$ obtained by applying elements $g \in G(B)$ to the hyperplanes $\P_1, ..., \P_s.$
If $\sigma_j$ is the reflection across $\P_j$ then $f \circ \sigma_j= -f.$ Decompose
$$f(x)=\sum_{m=0}^{\infty}c_m Y_m(x)$$
into Fourier series on $S^{n-1}.$  Since the space of spherical harmonics of degree $m$ is $\sigma_{j}$-invariant, $f \circ \sigma_j =-f$ implies
$Y_m \circ \sigma_{j}=-Y_m$ for all $c_m \neq 0.$  Therefore, if $h_m$ is the harmonic homogeneous polynomial such that $h_m \vert_{S^{n-1}}=Y_m$ then $h_m$ vanishes on any hyperplane $\P_j$ and hence is divisible by $ \langle x, b_j \rangle, \ j \in \mathbb N,$ where $b_j$ is a normal vector to the hyperplane $\P_j.$ Thus, the polynomial $h_m$ is divisible by infinite many linear functions  and hence $h_m=0.$ Since $m$ is arbitrary, $ f =0.$
\end{proof}
Notice, that according to Corollary \ref{C:finite} non-injective pairs $F_{a,b}=(F_a, F_b) \ (s=2)$  of Funk transforms are also characterized by the finiteness of the reflection group $G(\{a,b\} ).$ It is not clear, whether similar characterization remains true for $s>2.$

\section{Concluding remarks}
\begin{itemize}
\item The intertwining relations between the shifted transform $F_a$ and  the standard transforms $F_0$ or $\Pi_a,$
for which inversion formulas are well known, lead to corresponding inversion formulas for the shifted transforms (\cite{Q}, \cite{AgR}, \cite{AgR1}) defined on $a$-even functions (see the definition at the end of Section \ref{S:arbitrary}).
\item In \cite{AgR}, a reconstructing series is built for the paired transform $(F_a, F_b)$ with two interior centers. The reconstruction is
given by the Neumann series for the operator $W$ (\ref{E:W-def}) and converges in $L^p$ for $ 1 \leq p < \frac{n-1}{k-1} \ (k>1).$  The results of this article lead to the similar inversion formula in the general case of arbitrarily located centers. We hope to return to the reconstruction problem elsewhere.

\end{itemize}

\bigskip
{\bf Acknowledgements }  The author thanks Boris Rubin for useful discussions of the results of this article and editorial advices.

\bibliographystyle{amsplain}

\end{document}